\documentclass[a4paper,11pt,reqno]{amsart}
\usepackage{amssymb}
\usepackage{amsmath}
\usepackage{enumerate}
\numberwithin{equation}{section} \numberwithin{figure}{section}

\DeclareMathOperator{\Pic}{Pic} 
 
 \DeclareMathOperator{\rank}{rank}
 \DeclareMathOperator{\re}{Re}
   
\DeclareMathOperator{\vol}{vol} 

\newcommand{\Aone}{{\mathbf A}_1}

\newcommand{\tS}{{\widetilde S}}

\newcommand\xx{\mathbf{x}}
\newcommand\yy{\mathbf{y}}

\newcommand\mm{\mathbf{m}}
\newcommand{\dd}{\mathbf{d}}
\newcommand{\intd}{\mathrm{d}}
\newcommand{\DD}{\mathbf{D}}

\newcommand{\ee}{\mathbf{e}}

\newcommand{\kk}{\mathbf{k}}
\newcommand{\eeta}{\mbox{\boldmath$\eta$}}
\newcommand{\eepsilon}{\mbox{\boldmath$\epsilon$}}
\newcommand{\kkappa}{\mbox{\boldmath$\kappa$}}

\newcommand{\ddelta}{\mbox{\boldmath$\delta$}}
\newcommand{\xxi}{\mbox{\boldmath$\xi$}}
\newcommand{\regA}{\mathcal{A}}
\newcommand{\regR}{\mathcal{R}}

\newcommand\PP{\mathbb{P}}
\newcommand\Pone{{\PP^1}}
\newcommand\Ptwo{{\PP^2}}
\newcommand\ZZ{\mathbb{Z}}
\newcommand\NN{\mathbb{N}}
\newcommand\QQ{\mathbb{Q}}
\newcommand\RR{\mathbb{R}}

\newcommand{\Mob}{M\"{o}bius }

\newtheorem{lemma}{Lemma}
\newtheorem{theorem}[lemma]{Theorem}
\newtheorem{corollary}[lemma]{Corollary}

\numberwithin{lemma}{section}

\title[Manin's Conjecture]{Manin's conjecture for a singular quartic del Pezzo surface}
\author{\sc Daniel Loughran}

\begin{document}

\maketitle

\begin{abstract}
    We prove Manin's conjecture for a split singular quartic del Pezzo surface
    with singularity type $2\Aone$ and eight lines. This is achieved by
    equipping the surface with a conic bundle structure. To handle the sum
    over the family of conics, we prove a result of independent interest
    on a certain restricted divisor problem for four binary linear forms.
\end{abstract}

\section{Introduction}
For any projective variety $X \subset \PP^n$ over $\QQ$, we may define the height of a rational point $x \in X(\QQ)$ to be
$H(x)=\max\{|x_0|,\ldots,|x_n|\}$. Here we have choosen a representative $x = (x_0:\cdots:x_n)$ such that $(x_0,\ldots,x_n)$
is a primitive integer vector. A natural object of study in diophantine geometry is the following counting
function
$$N_U(B)=\#\{x \in U(\QQ):H(x) \leq B\},$$
defined for any $U \subset X$ and $B >0$. Manin and his collaborators (see \cite{FMT89} and \cite{BM90}) have formulated a series of conjectures on the asymptotic behaviour of these counting functions as $B \to \infty$. When $X$ is a Fano variety given by its anticanonical embedding, they have conjectured that there exists some $U \subset X$ open and a constant $c_X\neq0$ such that
$$N_U(B)\sim c_XB(\log B)^{\rho-1}$$
where $\rho=\rank \Pic(X)$, at least if the set of rational points on $X$ is Zariski dense. The constant $c_X$
has also received a conjectural adelic interpretation due to Peyre \cite{Pey95}.

There is a programme to try to prove Manin's conjecture for smooth and singular del Pezzo surfaces, the Fano varieties
of dimension two. See \cite{Bro07} or \cite[Table 1.]{DL10} for a reasonably up to date account of the
progress so far. In this paper we study the number of rational points of bounded height on a certain singular del Pezzo surface of degree four, given by the equations
$$S: x_0x_1=x_2^2, \quad x_3x_4=x_2(x_1-x_0),$$
in $\PP^4$. This surface has been chosen since it is a quartic del Pezzo surfaces with singularity type $2\Aone$ and
eight lines. Such surfaces are at the forefront of current methods, as a general philosophy
in the programme is that the milder the singularities, the more difficult Manin's conjecture is to prove.
It is easy to check the singularity type of $S$ -- the only singularities of $S$ are $(0:0:0:1:0)$ and $(0:0:0:0:1)$,
and these are both locally quadratic cones of the form $x_0x_1=x_2^2$.
It contains the following eight lines
\begin{align*}
    &x_2=x_i=x_j=0, \\
    &x_0=x_1,x_0=\pm  x_2, x_j=0,
\end{align*}
for any $i \in \{0,1\}$ and $j \in \{3,4\}$.
To see that there are no other lines, we appeal to the classification of
singular del Pezzo surfaces of degree four \cite[Prop. 5.6]{CT88}. A surface of singularity type $2\Aone$
may contain either eight or nine lines. In the case where it contains nine lines,
one of these lines joins the two singularities, and it is easy to
check that this is not the case here. Since each line is defined over $\QQ$,
we see that $S$ is a split singular quartic del Pezzo surface with singularity
type $2\Aone$ and eight lines. Note that a point $x \in S$ lies on a line if and only if
$x_0x_1x_2x_3x_4=0$.
Our result is as follows.
\begin{theorem} \label{thm:dp4asym}
    Let $U \subset S$ be the open subset of $S$ formed by removing all the lines.
    Then we have
    $$N_U(B) = c_{S}B(\log B)^5(1 + o(1))$$
    as $B \to \infty$, where $c_{S}$ is the leading constant as predicted by Peyre.
\end{theorem}
Note that we remove the lines since each line contributes roughly $B^2$ points
to the counting problem, obscuring the finer arithmetic of the surface.
An explicit expression for the leading constant can be found in Section \ref{sec:constant}.
The proof of the theorem is achieved by utilising a conic bundle structure on $S$.
This method was also used in \cite{BB08}, however it is in contrast
to many of the proofs of Manin's conjecture for other quartic del Pezzo surfaces, which have used
the associated universal torsor, see e.g. \cite{BB07}. When the singularity type of the surface in question is not
so mild, the universal torsor is often an open subset of a hypersurface in affine space. However the universal torsor for $S$
has many more equations, so the previous methods used for dealing with such surfaces would be harder to implement here.
The conic bundle structure on $S$ allows us to transform the problem of counting rational points
on $S$ to one of counting rational points on a family of conics, essentially given by
\begin{equation}
    xy=ab(b^2-a^2)z^2, \label{eqn:conic}
\end{equation}
for varying parameters $a$ and $b$. Counting the rational points of bounded height on any one individual conic is relatively simple,
the difficultly arises when we sum over all the conics in the family. To handle this sum we prove an
auxiliary result of independent interest in analytic number theory.
It concerns the asymptotic behaviour of a certain restricted divisor problem for four binary linear forms.
We postpone a precise statement of our result since it is of a technical nature, however a simple corollary is that
\begin{align*}
    &\sum_{ \xx \in \ZZ^2 \cap X\regR}\tau(L_1(\xx))\tau(L_2(\xx))\tau(L_3(\xx))\tau(L_4(\xx))
    \sim cX^2(\log X)^4,
\end{align*}
as $X \to \infty$. Here $\regR \subset \RR^2$ is some suitable region, $L_1,L_2,L_3,L_4$ are certain non-proportional binary linear forms and
$c=c(L_1,L_2,L_3,L_3,)$ is a constant. In our application
to counting points on conics, our binary linear forms are essentially $x_1,x_2,x_2-x_1$ and $x_2+x_1$, which geometrically
correspond to the discriminant of the family in question (\ref{eqn:conic}). Sums of the shape
$$
    \sum_{\xx \in \ZZ^2 \cap X\regR} \prod_{i=1}^nf\left(L_i(\xx)\right),
$$
for binary linear forms $L_1,\ldots,L_n$ and certain arithmetic functions $f$ have been considered before. The case
where $n=3$ and $f=\tau$ has been handled in \cite{Bro11}, and Heath-Brown considered the case where $n=4$ and $f=r$, the sum
of squares function. Our methods are similar to these and are based on the work of Daniel \cite{Dan99}, and
the case $n=4$ seems to be the limit of what these methods can achieve. There is however recent work of Matthiesen \cite{Mat11} in which she proves an asymptotic formula for arbitrary $n$ and $f=\tau$, using techniques from additive combinatorics.
However, this result is not sufficient for our purposes as the fact that we consider a \emph{restricted} divisor function
is essential to our proof of Manin's conjecture.

We note that Theorem \ref{thm:dp4asym} is related to, but does not follow from, the work of \cite{BBP10}, where
they prove Manin's conjecture for a family of Ch\^{a}telet surfaces, using the universal torsor approach.
The surfaces they consider are the minimal
desingularisations of a family of Iskovskikh surfaces \cite[Rem.2.3]{BBP10}, which are also
del Pezzo surfaces of degree four with singularity type $2\Aone$ and eight lines. However, for such surfaces the
two singularities are \emph{conjugate}, and thus these surfaces are not split. We also note that the case of singularity
type $2\Aone$ and nine lines can be handled using similar methods to what we use here, and it actually seems
to be easier than the eight lines case due to a simpler divisor problem arising.


The layout of this paper is as follows. Section two is dedicated to the above
mentioned restricted divisor problem. In the third section we gather numerous preliminary results
on lattice point counting and divisors problems, before using these results to prove Theorem \ref{thm:dp4asym} in Section four.

\textbf{Notation}: We use $\nu_p(x)$ to denote the $p$-adic valuation of a
rational number $x$.

\subsection{The leading constant} \label{sec:constant}
We now give a description of the leading constant $c_S$ appearing in Theorem \ref{thm:dp4asym}.
It agrees with the constant as predicted by Peyre \cite{Pey95}, and writing it down explicitly amounts
to a now standard calculation, see e.g. \cite{BB07}. If  $\tS$ denotes the minimal desingularisation of $S$, then since $S$ is split we have
$$c_S=\alpha(\tS)\tau_{\infty}\prod_p\tau_p,$$
where $\alpha(\tS)$ is the ``nef cone volume" and $\tau_v$ denotes the density of $S$ at the place $v$, with the necessary
convergence factor included.
By \cite[Lem. 2.3]{Lou10} we have
$$\tau_p=\left(1 - \frac{1}{p}\right)^6\left(1 + \frac{6}{p} + \frac{1}{p^2}\right),$$
for all primes $p$. Also \cite[Table 5]{Der07} tells us that $$\alpha(\tS) = \frac{1}{720}= \frac{1}{2^4 3^2 5}.$$
To calculate the density at the real place we use the Leray form of $S$ (see \cite[Sec. 5.2]{Pey95}), which is given by
$$\omega_L(S) = \frac{\intd x_0 \intd x_1 \intd x_3}{2(x_0x_1)^{1/2}x_3},$$
since $$\det\left(\begin{array}{ll}
                 \frac{\partial Q_1}{\partial x_2} &\frac{\partial Q_2}{\partial x_2} \\ &\\
                 \frac{\partial Q_1}{\partial x_4} &\frac{\partial Q_2}{\partial x_4} \\
                 \end{array}\right) = -2x_2x_3,$$
where $Q_1(\xx)=x_0x_1-x_2^2$ and $Q_2(\xx)=x_3x_4-x_2(x_1-x_0)$.
Note that $x_2^2=x_0x_1\geq0$, so the Leray form is well-defined. The density at the real place is then given by
\begin{align*}
    \tau_\infty&=\frac{1}{2}\int_{\{\xx \in \RR^5 : Q_1(\xx)=Q_2(\xx)=0, |x_0|,|x_1|,|x_2|,|x_3|,|x_4| \leq 1 \}} \omega_L(S).
\end{align*}
We can turn this integral into a  slightly more amenable form by taking advantage of certain automorphisms of the surface $S$. We already know that $x_0x_1\geq0$, however we may also assume that $x_0,x_1\geq0$. Indeed the above integral is invariant under the automorphism which negates $x_0,x_1$ and $x_4$. Similarly we may assume that $x_1\geq x_0$, since we may swap them and again negate $x_4$. Finally, we may
negate $x_3$ and $x_4$ to assume $x_3$ is positive, and negate $x_2$ and $x_4$ to assume that $x_2$ is positive. Hence
\begin{align*}
    \tau_\infty &= 4\int_{\{\xx \in \RR^3 : 0< x_0/x_1, x_1, x_3,x_0x_1(x_1-x_0)^2/x_3^2 \leq 1 \}} \frac{\intd x_0 \intd x_1 \intd x_3}{(x_0x_1)^{1/2}x_3}.
\end{align*}

\section{A restricted divisor problem} \label{sec:div}
We now describe in detail the restricted divisor problem which we handle in this paper. As mentioned in the introduction,
this result will be used to handle the sum over the family of conics on $S$.
Fix a lattice $\Lambda \subset \ZZ^2$, equipped with the usual Euclidean inner product. Let $\regR \subset \RR^2$ be region, that is, a compact set with a continuous
piecewise differentiable boundary $\partial\regR$, whose length we denote by $|\partial\regR|$. Also let $L_1(\xx),\ldots,L_4(\xx) \in \QQ[\xx]$ be linear forms, no two of which are proportional and which satisfy $L_i(\xx) \in \ZZ$ for all $\xx \in \Lambda$ and $L_i(\xx) \geq0$ for all $\xx \in \regR$ $(i=1,2,3,4)$. We then define
$$r=\sup_{\xx \in \regR}\{L_1(\xx),L_2(\xx),L_3(\xx),L_4(\xx),|x_1|,|x_2|\}.$$
Next let $X\geq1$ and for simplicity we assume that our region satisfies
satisfies $|\partial X\regR| \ll rX$, where we write $X\regR = \{\xx \in \RR^2:\xx/X \in \regR\}$. Such a condition is automatically satisfied if $\regR$ is convex,
for example.
Finally let $V=V(X) \subset [0,1]^4$ be a non-empty compact set that is cut out by a bounded number of hyperplanes each with bounded coefficients. Then, we are interested in getting an asymptotic formula for the following sum
$$S(X;V) = \sum_{\xx \in \Lambda \cap X\regR}
\tau(L_1(\xx),L_2(\xx),L_3(\xx),L_4(\xx);V),$$
as $X \to \infty$.  Here
$$
\tau(L_1(\xx),L_2(\xx),L_3(\xx),L_4(\xx);V)=
\#\left\{\dd \in \NN^4: d_i|L_i(\xx), \ddelta \in V\right\}, \quad \ddelta = \left(\frac{\log d_i}{\log rX}\right)_{i=1,2,3,4}.\\
$$
Note that our choice of $r$ ensures that $\tau(L_1(\xx),L_2(\xx),L_3(\xx),L_4(\xx);[0,1]^4)$ is
simply a fourfold product of the usual divisor function. In fact we shall soon see that by considering $V \subsetneq [0,1]^4$,
only the leading constant changes in the asymptotic formula,
namely $S(X;V)= S(X;[0,1]^4)(\vol V + o(1))$ as $X\to \infty$.
To state the result that we prove, let
\begin{equation} \label{def:rho}
    \rho(\dd)=\frac{\det \Lambda(\dd)}{\det \Lambda}, \quad
    \Lambda(\dd)=\{\xx \in \Lambda:d_i | L_i(\xx),(i=1,2,3,4)\},
\end{equation}
where we define the determinant of a lattice to be the measure of any
fundamental domain. Next choose the minimum $c_i \in \NN$ such that
$L_i(\xx)= \ell_i(\xx)/c_i$, where $\ell_i(\xx) \in \ZZ[\xx]$, and let $\Delta \in \ZZ$ be
the product of the resultants of the pairs of linear forms $\ell_i$ and $\ell_j$ for $i \neq j$. Note that $p| \Delta$
if and only if the form $\ell_1\ell_2\ell_3\ell_4$ has singular reduction modulo $p$.

\begin{theorem}\label{thm:resdivisor}
    Let $X \geq 1$. Then we have
    $$S(X;V) = \frac{C_\infty \prod_p C_p}{\det \Lambda}X^2(\log X )^4 + O_{L_1,L_2,L_3,L_4,r,\Lambda}(X^2(\log X)^3 \log \log X)$$
    as $X \to \infty$, where
    $$
        C_\infty = \vol{\regR} \vol V, \quad
        C_p  = \left(1 - \frac{1}{p}\right)^4
        \left(\sum_{\kk \in \ZZ_{\geq 0}^4}
        \frac{1}{\rho(p^{k_1},p^{k_2},p^{k_3},p^{k_4})} \right).
    $$
    Moreover $\prod_p |C_p| \ll_\varepsilon (\Delta \det \Lambda)^{\varepsilon}$ for any $\varepsilon > 0$.
\end{theorem}
Note that the error term here is independent of $V$, which is essentially because we may use the upper bound $\tau(L_1(\xx),L_2(\xx),L_3(\xx),L_4(\xx);V) \leq \tau(L_1(\xx))\tau(L_2(\xx))\tau(L_3(\xx))\tau(L_4(\xx))$
to handle each error term.
For the application we have in mind, we need a related result. Namely, let $V'=V'(X) \subset [0,1]^5$ be a
non-empty compact set that is cut out by a bounded number of hyperplanes each with bounded coefficients. Then we define
$$S'(X;V') = \sum_{\xx \in \Lambda \cap X\regR}\frac{\tau'(L_1(\xx),L_2(\xx),L_3(\xx),L_4(\xx);V')}{\max\{x_1,x_2\}^2},$$
where now
$$
\tau'(L_1(\xx),L_2(\xx),L_3(\xx),L_4(\xx);V')=
\#\left\{\dd \in \NN^4: d_i|L_i(\xx), \left(\ddelta,\frac{\log \max\{|x_1|,|x_2|\}}{\log rX}\right) \in V'\right\}.
$$
Note that the important difference here is that we are allowing the restriction placed on the divisors to depend on the varying
parameter $\xx$. It is then relatively simple to get an asymptotic formula for $S'(X;V')$ using
Theorem \ref{thm:resdivisor}.
\begin{corollary}\label{cor:resdivisor}
    Let $X \geq 1$ and let $\chi_{V'}$ denote the characteristic function of the set $V'$. Then we have
    \begin{align*}
        S'(X;V')=\frac{2C'_\infty \prod_p C_p}{\det \Lambda}(\log X)^5 +
        O_{L_1,L_2,L_3,L_4,r,\Lambda}((\log X)^4 \log \log X),
    \end{align*}
    as $X \to \infty$, where
    \begin{align*}
        C'_\infty &= \vol{\regR} \int_{\substack{u \in [0,1] \\ \eeta \in [1,u]^4}} \chi_{V'}(\eeta,u) \mathrm d \eeta \mathrm d u,
    \end{align*}
    and the $C_p$ are as given in Theorem \ref{thm:resdivisor}.
\end{corollary}

\subsection{Some multiplicative functions}
Before we begin the proof of Theorem \ref{thm:resdivisor},
we briefly collect some facts about the function
$\rho(\dd)=\det \Lambda(\dd)/\det \Lambda,$
as defined in (\ref{def:rho}), and some related functions. First
note that $\rho$ is a multiplicative function. Indeed, we have the obvious
equality $\rho(\dd)=\#(\Lambda/\Lambda(\dd))$, and the Chinese remainder theorem
gives an isomorphism $\Lambda/\Lambda(\dd\ee) \cong \Lambda/\Lambda(\dd) \times \Lambda/\Lambda(\ee)$
for any $\dd,\ee \in \NN^4$ such that $(d_1d_2d_3d_4,e_1e_2e_3e_4)=1$.
\begin{lemma} \label{lem:rho}
    For any $e_1,e_2,e_3,e_4 \geq 0$, let $\sigma$ be the permutation such that
    $e_{\sigma(1)} \geq
    e_{\sigma(2)}\geq e_{\sigma(3)} \geq e_{\sigma(4)}$.
    Then for any prime $p$ we have
    $$\rho(p^{e_1},p^{e_2},p^{e_3},p^{e_4}) \left\{
        \begin{array}{ll}
                = p^{e_{\sigma(1)} + e_{\sigma(2)}}, \quad &p \nmid \Delta \det \Lambda, \\
                \geq p^{\max\{e_{\sigma(1)} + e_{\sigma(2)} - \lambda_p - 2\delta_p,0\}}, \quad &p | \Delta \det \Lambda,
        \end{array}\right.$$
    where $\lambda_p = \nu_p(\det \Lambda)$ and $\delta_p = p(\Delta)$.
\end{lemma}
\begin{proof}
    We begin the proof with a preliminary result. To simplify notation, let $p^{\ee}=(p^{e_1},\ldots,p^{e_4})$ and
    consider the lattice $\Gamma_{\dd} = \{ \xx \in \ZZ^2: d_i | \ell_i(\xx)\}$,
    where as before we have chosen the minimum $c_i \in \NN$ such that
    $L_i(\xx)= \ell_i(\xx)/c_i$ and $\ell_i(\xx) \in \ZZ[\xx]$.
    Then I claim that
    \begin{equation}\label{eqn:Gamma}
    \begin{array}{ll}
       \det \Gamma_{p^{\ee}} \left\{
        \begin{array}{ll}
                = p^{e_{\sigma(1)} + e_{\sigma(2)}}, \quad &p \nmid \Delta, \\
                \geq p^{\max\{e_{\sigma(1)} + e_{\sigma(2)}-2\delta_p,0\}}, \quad &p | \Delta.
        \end{array}\right.
    \end{array}
    \end{equation}
    Indeed for $p \nmid \Delta$, as in \cite[p.13]{HB03} we find that $p^{e_i} | \ell_i(\xx)$ for $i=1,2,3,4$ is equivalent
    to$$ p^{e_{\sigma(2)}} | \xx, \quad p^{e_{\sigma(1)}} | \ell_{\sigma(1)}(\xx).$$
    Thus $\Gamma_{p^{\ee}}$ has determinant $p^{e_{\sigma(1)} + e_{\sigma(2)}}$. For all other primes $p$,
    note that $x \in \Gamma_{p^{\ee}}$ implies
    $p^{e_{\sigma(2)}} | \Delta \xx$ and
    $p^{e_{\sigma(1)}} | \ell_{\sigma(1)}(\xx)$. Now, $\ell_{\sigma(1)}$ is not necessarily
    primitive, however any fixed divisor of $\ell_{\sigma(1)}$ must divide $\Delta$, so we deduce that
    \begin{equation}
        p^{e_{\sigma(2)}} | p^{\delta_p} \xx, \quad p^{e_{\sigma(1)}} | p^{\delta_p}\ell^*_{\sigma(1)}(\xx), \label{eqn:Deltalattice}
    \end{equation}
    where $\ell^*_{\sigma(1)}$ is a primitive linear form. If $e_{\sigma(1)} \leq \delta_p$, the lattice
    given by (\ref{eqn:Deltalattice}) clearly has determinant $1$. Similarly if $e_{\sigma(2)} \geq \delta_p$, then the
    lattice has determinant $p^{e_{\sigma(1)} + e_{\sigma(2)} - 2\delta_p}$. Finally, if
    $e_{\sigma(1)} > \delta_p$ and $e_{\sigma(2)} \leq \delta_p$, then the lattice
    given by (\ref{eqn:Deltalattice}) has determinant
    $p^{e_{\sigma(1)} - \delta_p} \geq p^{e_{\sigma(1)} + e_{\sigma(2)} - 2\delta_p}$, thus proving
    (\ref{eqn:Gamma}).

    We now use (\ref{eqn:Gamma}) to prove the lemma.
    Note that $c_i | \det \Lambda$ for $i=1,2,3,4$, since each $L_i$ takes only integral
    values on $\Lambda$. Hence for any $p \nmid \Delta \det \Lambda$,
    we have $\Lambda({p^{\ee}})=\Lambda \cap \Gamma_{p^{\ee}}$.
    The Chinese remainder theorem implies that
    $\det\Lambda({p^{\ee}})=\det \Lambda \det \Gamma_{p^{\ee}}$, so the result follows from (\ref{eqn:Gamma}).
    For all other primes $p$, it is clear that $\Lambda({p^{\ee}})$ is still a sublattice of $\Lambda$ and $\Gamma_{p^{\ee}}$,
    so $\det \Lambda({p^{\ee}}) \geq [\det \Lambda, \det \Gamma_{p^{\ee}}]$. Note however that
    $(\det \Lambda, \det \Gamma_{p^{\ee}}) \leq \lambda_p$ as  $p$ is the only prime
    dividing $\det \Gamma_{p^{\ee}}$. Thus, by (\ref{eqn:Gamma}) we have
    $$\rho({p^{\ee}}) \geq \frac{\det \Gamma_{p^{\ee}}}{(\det \Lambda, \det \Gamma_{p^{\ee}})}
    \geq  p^{e_{\sigma(1)} + e_{\sigma(2)} - \lambda_p -2\delta_p}.$$
\end{proof}
For any $\kk \in \NN^4$ let
\begin{equation}
    \upsilon(\kk)=\sum_{d_i|k_i}\frac{d_1d_2d_3d_4}{\rho(\dd)}\mu\left(\frac{k_1k_2k_3k_4}{d_1d_2d_3d_4}\right). \label{def:upsilon}
\end{equation}
We have defined $\upsilon$ via a higher dimensional analogue of the usual Dirichlet convolution, in such a way that it
is small in general. The next lemma makes this more precise.

\begin{lemma} \label{lem:upsilon}
    Let
    $$\Upsilon(s)=\sum_{\kk \in \NN^4}\frac{\upsilon(\kk)}{(k_1k_2k_3k_4)^s},$$
    be the Dirichlet series corresponding to $\upsilon$, as defined by (\ref{def:upsilon}).
    Then $\Upsilon(s)$ is absolutely convergent on the half-plane $\re(s) >5/6$.
    Moreover for any $\varepsilon > 0$ we have
    $$\Upsilon(1)=\prod_p C_p \ll_\varepsilon (\Delta \det \Lambda)^{\varepsilon},$$
    where the $C_p$ are as given in Theorem \ref{thm:resdivisor}.
\end{lemma}
    \begin{proof}
    Let $\varepsilon>0$ and let $s \geq 5/6+\varepsilon$. Then by multiplicativity we have
    \begin{align*}
        \sum_{\kk} \frac{|\upsilon(\kk)|}{{(k_1k_2k_3k_4)^{s}}}
        &= \prod_p\left(\sum_{\substack{\ee \in \ZZ^4_{\geq0}}}
        \frac{|\upsilon(p^{e_1},p^{e_2},p^{e_3},p^{e_4})|}{p^{(e_1+e_2+e_3+e_4)s}}\right). \\
    \end{align*}
    However, when $p \nmid \Delta \det \Lambda$ and $0 < e_1+e_2+e_3+e_4 \leq 2$, Lemma
    \ref{lem:rho} implies that
    $\rho(p^{e_1},p^{e_2},p^{e_3},p^{e_4})=p^{e_1 + e_2 + e_3 + e_4}$
    and hence
    $\upsilon(p^{e_1},p^{e_2},p^{e_3},p^{e_4})=0$.
    It follows that the contribution from $p \nmid \Delta \det \Lambda$ is bounded above by
    \begin{align*}
         &\prod_{p \nmid \Delta \det \Lambda}\left(1 + \sum_{\substack{e \geq 3}}
        \frac{1}{p^{es}}\sum_{\substack{e_1+e_2+e_3+e_4=e}}
        \frac{e^4p^{e}}{\rho(p^{e_1},p^{e_2},p^{e_3},p^{e_4})}\right)\\
        &\ll \prod_p\left(1 + \sum_{\substack{e \geq 3}}
        \frac{e^8}{p^{e(s-1/2)}}\right) \\
        &\ll_{\varepsilon} \prod_p\left(1 + \frac{1}{p^{1 + \varepsilon}}\right) \ll_{\varepsilon} 1. \\
    \end{align*}
    Similarly, those primes $p | \Delta \det \Lambda$ contribute $\ll_{\varepsilon,\Delta, \det \Lambda} 1$.
    Next by the definition of $\upsilon$, for $\re(s)>5/6$ we have
    $$\Upsilon(s)=\prod_p\left(1 - \frac{1}{p^s}\right)^4
            \left(\sum_{\kk \in \ZZ_{\geq 0}^4}\frac{p^{(k_1+k_2+k_3+k_4)(1-s)}}{\rho(p^{k_1},p^{k_2},p^{k_3},p^{k_4})} \right).$$
    Thus the equality $\Upsilon(1)=\prod_p C_p$ is clear. To show the upper bound, by Lemma \ref{lem:rho}
    we have
    \begin{align*}
        \Upsilon(1) &\ll \prod_{p | \Delta \det \Lambda}\left(1 - \frac{1}{p}\right)^4
        \left(\sum_{\kk \in \ZZ_{\geq 0}^4}\frac{1}{\rho(p^{k_1},p^{k_2},p^{k_3},p^{k_4})} \right) \\
        &\ll \prod_{p | \Delta \det \Lambda}\left((\lambda_p + 2\delta_p)^4 + O(1)\right)\\
        &\ll_{\varepsilon} (\Delta \det \Lambda)^{\varepsilon}.
    \end{align*}
\end{proof}

\subsection{Proof of Theorem \ref{thm:resdivisor}}
In what follows all errors terms are
implicitly allowed to depend on the linear forms $L_1,L_2,L_3,L_4$, the number $r$ and the lattice $\Lambda$.
We begin by showing that we need only sum over the smaller divisors of the linear forms.

\begin{lemma}
For any $\varepsilon>0$ we have
$$S(X;V)=\sum_{\mm \in \{\pm 1\}^4}S^{\mm}(X;V) + O_\varepsilon(X^{3/2+\varepsilon}),$$
as $X \to \infty$, where
$$S^{\mm}(X;V) = \sum_{\xx \in \Lambda \cap X\regR}
 \#\left\{\dd \in \NN^4:
 \begin{array}{ll}
     d_i|L_i(\xx), d_i \leq \sqrt{L_i(\xx)} \\
     \DD^{\mm}(\ddelta,\xxi) \in V
 \end{array}\right\},$$
and
\begin{align*}
    \DD^{\mm}(\ddelta,\xxi)&=\mm\ddelta + (1-\mm)\xxi/2,
    \quad \xxi= \left(\frac{\log{L_i(\xx)}}{\log rX}\right)_{i=1,2,3,4}. \\
\end{align*}
\end{lemma}
\begin{proof}
    To get the main term we use a variant of the classical Dirichlet hyperbola method, namely if
    $d_i > \sqrt{L_i(\xx)}$, we replace $d_i$ by $L_i(\xx)/d_i$. The error term is then
    made up of those terms where $d_i = \sqrt{L_i(\xx)}$ for some $i=1,2,3,4$,
    each of which is handled in a similar manner. For
    example the contribution from where $L_4(\xx)$ is a square is
    \begin{align*}
        &\sum_{\substack{\xx \in \Lambda \cap X\regR \\ L_4(\xx)=\Box}}
        \tau(L_1(\xx))\tau(L_2(\xx))\tau(L_3(\xx)) \\
        &\ll_\varepsilon X^\varepsilon \sum_{n \leq \sqrt{X}}\#\{\xx \in \ZZ^2: ||\xx||\ll X, L_4(\xx)=n^2\}
         \ll_\varepsilon X^{3/2 + \varepsilon}.
    \end{align*}
\end{proof}
For now we consider fixed $\mm$.
After changing the order of summation, we have
$$S^{\mm}(X;V)= \sum_{d_i \leq r\sqrt{X}}\#( \xx \in \Lambda(\dd) \cap
\regR^{\mm}(\dd;X)),$$ where
$$
    \regR^{\mm}(\dd;X)= \{ \xx \in X\regR: d_i \leq
    \sqrt{L_i(\xx)}, \DD^{\mm}(\ddelta,\xxi) \in V \},
$$
and $\Lambda(\dd)$ is given by (\ref{def:rho}).
Large divisors will become problematic for us, so we sum over these separately. Write
$$ S^{\mm}_0(X;V)= \sum_{\substack{d_i \leq r\sqrt{X}  \\ d_4 \geq Y}}
    \#( \xx \in \Lambda(\dd) \cap \regR^{\mm}(\dd;X)), \quad
    S^{\mm}_1(X;V)=S^{\mm}(X;V)-S^{\mm}_0(X;V),
$$
where $Y \leq r\sqrt{X}$ is some parameter to be chosen later. We may handle
$S^{\mm}_1(X;V)$ with the following ``level of distribution" result.
\begin{lemma}
    Let $X \geq 1$ and $Q_1,Q_2,Q_3,Q_4 \geq 2$. Write $$Q= \max_i Q_i \text{ and } P=Q_1Q_2Q_3Q_4.$$ Then
    there is an absolute constant $A>0$ such that
    \begin{align*}
        \sum_{d_i \leq Q_i}& \left|\#(\Lambda(\dd) \cap \regR^{\mm}(\dd;X)) - \frac{\vol{\regR^{\mm}(\dd;X)}}{\det \Lambda}\right|
        \ll (XP^{1/2} + XQ + P)(\log Q)^A.
    \end{align*}
\end{lemma}
\begin{proof}
    This follows from \cite[Lem. 2.1]{HB03}, whose proof is a minor modification of the argument of Daniel \cite[Lem. 3.2]{Dan99}. Note that we work in slightly more generality than Heath-Brown, as our region $\regR^{\mm}(\dd;X)$
    is not neccessarily convex. However, the important point is that we still have the neccessary upper bound
    $|\partial \regR^{\mm}(\dd;X)| \ll rX$, uniformly with respect to $V$. Indeed, this follows from our simplifying
    assumption on our region that $|\partial X\regR| \ll rX$, and also the fact that $V$ is cut out by a bounded number of hyperplanes each with bounded coefficients.
\end{proof}
Hence if we take $Y= r\sqrt{X}/(\log X)^{2A}$, we deduce that
$$S^{\mm}_1(X;V) = \sum_{\substack{d_i \leq r\sqrt{X}\\d_4 \leq Y}
} \frac{\vol(\regR^{\mm}(\dd;X))}{\det \Lambda(\dd)} +
O(X^2).$$
We get an upper bound for $S^{\mm}_0(X;V)$ with the next lemma.
\begin{lemma}
    Let $X \geq 1$. Then we have
    $$S^{\mm}_0(X;V) \ll X^2(\log X)^3(\log \log X),$$
    as $X \to \infty$.
\end{lemma}
\begin{proof}
    We begin by defining a kind of
    generalised divisor function, defined multiplicatively for any prime $p$ by
    $$
    \eth_3(p^a)=    \left\{\begin{array}{ll}
                        2,& \quad a=1,\\
                        (a+1)^{3},& \quad a \neq 1.
                    \end{array}\right.
    $$
    We will meet this function later on in a more general context in Section \ref{sec:Useful_results}. Notice that
    \begin{align*}
        S^{\mm}_0(X;V) &\ll \sum_{Y \leq d_4 \leq r\sqrt{X}}\sum_{\substack{\xx \in \Lambda \cap
        X\regR \\ d_4 | L_4(\xx)}}
        \tau(L_1(\xx))\tau(L_2(\xx))\tau(L_3(\xx)) \\
        &\ll \sum_{Y \leq d \leq r\sqrt{X}}\sum_{\substack{x \ll X\\ y \ll X/d}}
        \eth_3(\ell'_1(x,dy)\ell'_2(x,dy)\ell'_3(x,dy)),
    \end{align*}
    where $\ell'_i$ is the linear form obtained from $\ell_i$ by the change of variables $x_2 \mapsto \ell_4(\xx)$, for $i=1,2,3$. We now appeal to \cite[Thm. 1]{BB06}, which is a general result on upper bounds for sums of arithmetic functions taking values in binary forms. Let $\Delta'(d)$ denote the discriminant of the form $F'(x,y)=\ell'_1(x,dy)\ell'_2(x,dy)\ell'_3(x,dy)$, and let $\psi(d)=\prod_{p|d}(1 + 1/p)$. Then \cite[Thm. 1]{BB06} allows
    us to deduce that
    \begin{align*}
        S^{\mm}_0(X;V) &\ll \sum_{Y \leq d \leq r\sqrt{X}} \frac{\psi(\Delta'(d))X^2(\log X)^3}{d} \\
                     &\ll \sum_{Y \leq d \leq r\sqrt{X}} \frac{\psi(d)X^2(\log X)^3}{d}
                     \ll X^2(\log X)^3(\log \log X),
    \end{align*}
    as required.
\end{proof}
Hence we have
\begin{equation} \label{eqn:Sfinal}
    S(X;V) =\frac{1}{\det \Lambda}\sum_{\mm \in \{\pm 1\}^4}\sum_{\substack{d_i \leq r\sqrt{X}\\ d_4 \leq Y}
    } \frac{\vol(\regR^{\mm}(\dd;X))}{\rho(\dd)}
    + O(X^2(\log X)^{3}(\log \log X)),
\end{equation}
where $\rho$ is given by (\ref{def:rho}).
Note that
\begin{align*}
    \sum_{\substack{d_i \leq r\sqrt{X}\\ d_4 \leq Y}}\frac{\vol(\regR^{\mm}(\dd;X))}{\rho(\dd)}
    & = \sum_{\substack{d_i \leq r\sqrt{X}\\ d_4 \leq Y}} \frac{\sum_{k_i | d_i}
    \upsilon(\kk)\vol(\regR^{\mm}(\dd;X))}{d_1d_2d_3d_4} \\
    & = \sum_{\substack{k_i \leq r\sqrt{X}}} \frac{\upsilon(\kk)}{k_1k_2k_3k_4} E_{\kk}^{\mm}(X;V).
\end{align*}
Here $\upsilon$ is given by (\ref{def:upsilon}) and
$$E_{\kk}^{\mm}(X;V)= \sum_{\substack{e_i \leq r\sqrt{X}/k_i \\ e_4 \leq Y/k_4}}\frac{\vol(\regR^{\mm}(\mathbf{ek};X))}{e_1e_2e_3e_4},$$
where we write $\ee\kk=(e_1k_1,e_2k_2,e_3k_3,e_4k_4)$.
We handle this inner sum with the following lemma.

\begin{lemma} \label{lem:sumvolR}
    Let $\kk \in \NN^4$ be such that $1 \leq k_i \leq r\sqrt{X}$ for $i=1,2,3,4$. Then for any $\varepsilon >0$ we have
    \begin{align*}
        E_{\kk}^{\mm}(X;V) =&\frac{C_\infty}{2^4} X^2(\log X)^4\left(1+ O_\varepsilon\left(\frac{||\kk||^{\varepsilon}(\log \log X)}{\log X}\right)\right),
    \end{align*}
    where $C_\infty$ is as in Theorem \ref{thm:resdivisor}.
\end{lemma}
\begin{proof}
    In what follows let
    \begin{align*}
        \eepsilon &= \left(\frac{\log e_i}{\log rX}\right)_{i=1,2,3,4},
        \quad \kkappa = \left(\frac{\log k_i}{\log rX}\right)_{i=1,2,3,4}.
    \end{align*}
    Then we have
    $$E_{\kk}^{\mm}(X;V)
    =\int_{\xx \in X\regR} \sum_{\substack{e_i \leq \sqrt{L_i(\xx)}/k_i \\ e_4 \leq Y/k_4
    \\ \DD^{\mm}(\eepsilon + \kkappa,\xxi) \in V}}
    \frac{\mathrm d \xx}{e_1e_2e_3e_4}.$$
    However, this simplifies to
    \begin{align*}
        E_{\kk}^{\mm}(X;V) &=\int_{\xx \in X\regR} \sum_{\substack{e_i \leq \sqrt{rX} \\ \DD^{\mm}(\eepsilon + \kkappa,\xxi) \in V}}
        \frac{\mathrm d \xx}{e_1e_2e_3e_4}
        + O_\varepsilon(||\kk||^{\varepsilon}X^2(\log X)^3(\log \log X)).
    \end{align*}
    Indeed, we may assume that $|x_1|,|x_2| \geq rX/\log X$ with a satisfactory error. Then the contribution from
    $\sqrt{L_i(\xx)}/k_i \leq e_i \leq \sqrt{rX}$ for $i=1,2,3$ is bounded above by the given error term. We may
    also handle $e_4$ in a similar manner.
    Performing Euler-Maclaurin summation, we find that
    $$
        E_{\kk}^{\mm}(X;V)=\frac{(\log rX)^4}{2^4}
        \int_{\substack{\eepsilon \in [0,1] \\\xx \in X\regR}}\chi_V(\DD^{\mm}(\eepsilon + \kkappa,\xxi)) \mathrm d \eepsilon \mathrm d \xx + O_\varepsilon\left(||\kk||^{\varepsilon}X^2(\log X)^3(\log \log X)\right),
    $$
    where $\chi_V$ is the characteristic function of $V$. On making the change of variables $\xx \mapsto \xx/X$ and
    $\eeta=\DD^{\mm}(\eepsilon + \kkappa,\xxi)=\mm(\eepsilon + \kkappa) +(1-\mm)\xxi/2$, we see that
    $$
        \int_{\substack{\eepsilon \in [0,1]\\ \xx \in X\regR}}\chi_V(\DD^{\mm}(\eepsilon + \kkappa,\xxi)) \mathrm d \eepsilon \mathrm d \xx =X^2\vol \regR \vol V + O_\varepsilon\left(\frac{||\kk||^{\varepsilon}X^2}{\log X}\right).
    $$
    However, by definition we have $C_{\infty}= \vol \regR \vol V$, thus proving the lemma.
\end{proof}
We are now in a position to finish the proof of Theorem \ref{thm:resdivisor}. First we sum over $\mm$ in
(\ref{eqn:Sfinal}), then use Lemma \ref{lem:upsilon} and Lemma \ref{lem:sumvolR} to deduce that
\begin{align*}
    S(X;V)&=\frac{C_\infty}{\det \Lambda}\sum_{\substack{k_i \leq r\sqrt{X}}} \frac{\upsilon(\kk)}{k_1k_2k_3k_4}X^2(\log X)^4
    \left(1 + O_\varepsilon\left(\frac{ ||\kk||^{\varepsilon}\log \log X}{\log X}\right)\right)\\
    &=\frac{C_\infty \prod_p C_p}{\det \Lambda}X^2(\log X)^4 + O\left(X^2(\log X)^3\log \log X \right),
\end{align*}
on choosing $\varepsilon =1/12$, say.

\subsection{Proof of Corollary \ref{cor:resdivisor}}
In what follows all error terms are implicitly allowed to depend on the linear forms, the number $r$ and the lattice $\Lambda$.
We first note that we have the identity
$$\frac{1}{\max \{x_1,x_2\}}=2\int_{\max \{x_1,x_2\}}^X \frac{\intd t}{t^3} + \frac{1}{X^2}.$$
Applying this we find
$$
S'(X;V') = \int_1^X \frac{S''(t;V')}{t^3} \intd t
+ O((\log X)^4)$$
where $S''(t;V')= \sum_{\xx \in \Lambda \cap t\regR}\tau'(L_1(\xx),L_2(\xx),L_3(\xx),L_4(\xx);V').$
In order to handle this sum using Theorem \ref{thm:resdivisor},
we need to remove the dependence on $\xx$. Our aim therefore is to replace the condition
$(\ddelta,(\log{\max\{x_1,x_2\}})/(\log rX)) \in V'$ by $(\ddelta,(\log{t})/(\log rX)) \in V'$.
To do this, for any $C \in \RR$ let
$$V_C(t)=\left\{\eeta \in [0,1]^4:\left(\eeta,1 + \frac{C\log \log t}{\log t}\right)
                 \in [0,1]^5 \cap \frac{\log rX}{\log rt}V'\right\}.$$
Then on noticing that we may assume that $|x_1|,|x_2| \geq rt/\log t$ with a suitable error,
we have the bounds
$$S(t;V_{-C}(t)) + O(t^2(\log t)^3)\leq S''(t;V') \leq S(t;V_{C}(t)) +  O(t^2(\log t)^3),$$
for some constant $C \geq 0$. However we clearly have $\vol V_{\pm C}(t)= \vol V_0(t) + O((\log \log t)/(\log t))$
for $t \gg 1$, hence
applying Theorem \ref{thm:resdivisor} we deduce that
\begin{align*}
S'(X;V')&=  \frac{\vol{\regR} \prod_p{C_p}}{\det \Lambda} \int_1^X  \frac{\vol V_{0}(t)(\log rt)^4}{t} \mathrm d t
+ O((\log X)^4 \log \log X).
\end{align*}
The proof of the corollary is then complete on noticing that
\begin{align*}
\int_1^X  \frac{\vol V_{0}(t)(\log rt)^4}{t} \mathrm d t
 &=  (\log rX)^4 \int_1^X  \frac{1}{t}
\int_{\eeta \in [1,\log rt/\log rX]^4} \chi_{V'}(\eeta,\log rt /\log rX) \mathrm d \eeta \mathrm d t \\
& = (\log X)^5  \int_{\substack{u \in [0,1] \\ \eeta \in [1,u]^4}} \chi_{V'}(\eeta,u) \mathrm d \eeta \mathrm d u + O((\log X)^4).\\
\end{align*}

\section{Useful results} \label{sec:Useful_results}
Before we begin the proof of Theorem \ref{thm:dp4asym}, we gather some technical results on lattice point counting
and upper bounds for certain divisor problems.
\subsection{Lattice point counting}
The emphasis on the results in this section is their \emph{uniformity} with respect to the chosen lattices and regions.
Our first result concerns counting \emph{non-zero} lattice points in planar domains. Before we state it, recall that
the first successive minima $\lambda_1$ of a lattice $\Lambda$ is defined to be the length of the shortest non-zero vector in $\Lambda$.

\begin{lemma} \label{plem:lattice}
    Let $X>0$. Let $\Lambda \subset \RR^2$ be a lattice with first successive minima $\lambda_1$ and
    suppose that $\regR \subset \RR^2$ is a region such that $(0,0) \in \regR$. Then
    $$\#\{ \xx \in \Lambda \cap X\regR: \xx \neq (0,0)\} = \frac{\vol{X\regR}}{\det{\Lambda}} +
    O\left(\frac{|\partial X\regR|}{\lambda_1}\right).$$
\end{lemma}
\begin{proof}
    The well-known method of counting lattice points in planar domains yields the estimate
    \begin{equation}
         \#\{\xx \in \Lambda \cap X\regR\} = \frac{\vol{X\regR}}{\det \Lambda}
        + O\left(\frac{|\partial X\regR|}{\lambda_1}+1\right). \label{eqn:Danlattice}
    \end{equation}
    If $1 \ll |\partial X\regR|/\lambda_1$, then the proof of the lemma follows immediately from (\ref{eqn:Danlattice}). Otherwise,
    suppose that  $|\partial X\regR| < \lambda_1$ and let $r(X)=\sup_{\xx \in X\regR}{||\xx||}.$ Then since the geodesics
    in $\RR^2$ are exactly the lines, we see that $r(X) \leq |\partial X\regR| < \lambda_1$, and hence there are no
    non-zero lattice points in $X\regR$. So in order for the statement of the lemma to be true in this case,
    it suffices to show that the error term dominates the main term. However we have
    $$\frac{\vol{X\regR}}{\det{\Lambda}}
    \ll \frac{(r(X))^2}{\det{\Lambda}} \ll \left(\frac{|\partial X\regR|}{\lambda_1}\right)^2 \ll 1.$$
    Hence
    $$\frac{\vol{X\regR}}{\det{\Lambda}}
    \ll \sqrt{\frac{\vol{X\regR}}{\det{\Lambda}}} \ll \frac{r(X)}{\lambda_1} \ll \frac{|\partial X\regR|}{\lambda_1},$$
    as required.
\end{proof}
For the next result we assume that $\regR$ is a ``box". Namely, there are some $r_1,r_2\geq 0 $ such
that
$$\regR =\{\xx \in \RR^2:  0 \leq x_i \leq r_i, (i=1,2)\}.$$

\begin{lemma} \label{plem:boxlattice}
    Let $X>0$ and let $\regR$ be a box. Then for any lattice $\Lambda \subset \ZZ^2$ we have
    $$\#\{ \xx \in \Lambda \cap X\regR: (x_1,x_2)=1\} \ll \frac{\vol{X\regR}}{\det{\Lambda}}+1.$$
    Next assume that $\Lambda=\{\xx \in \ZZ^2 : q_1|x_1,q_2|x_2\}$ for some $q_1,q_2 \in \NN$.
    Then
    $$\#\{ \xx \in \Lambda \cap X\regR: x_1x_2 \neq0\} \ll \frac{\vol{X\regR}}{\det \Lambda}.$$
\end{lemma}
\begin{proof}
    The first part of the lemma follows from \cite[Lem. 2]{HB84}, after bounding $\regR$ by a suitable ellipse.
    The second part is simple as the number of lattice points in question is clearly bounded above by $X^2r_1r_2/q_1q_2.$
\end{proof}

We finish with a result of Browning and Heath-Brown \cite[Cor. 2]{BHB07} on uniform upper bounds
for the number of points on conics.

\begin{lemma}\label{plem:conic}
Let $C$ be a non-singular ternary quadratic form. Let $\Delta$ denote
the determinant of the associated matrix, and let
$\Delta_0$ be the greatest common divisor of the $2 \times 2$ minors. Then we have
$$\#\left\{
    \xx \in \ZZ^3:
    \begin{array}{ll}
        C(\xx)=0, (x_1,x_2,x_3)=1 \\
        |x_i|\leq B_i, (i=1,2,3)
    \end{array} \right\}
    \ll \tau(|\Delta|)\left(1 + \frac{B_1B_2B_3 \Delta_0^{3/2}}{|\Delta|}\right)^{1/3}$$
\end{lemma}

\subsection{Divisor problems}
In this section we gather numerous results on upper bounds for certain divisor sums in two variables.
For any $k \in \NN$, we shall be interested in the following generalised divisor function, defined multiplicatively
for any prime $p$ by
\begin{equation} \label{def:eth}
    \eth_k(p^a)=    \left\{\begin{array}{ll}
                        2,& \quad a=1,\\
                        (a+1)^{k},& \quad a \neq 1.
                    \end{array}\right.
\end{equation}
We list the following simple properties of $\eth_k$ to clarify the relationship between it and the usual divisor function $\tau$.
\begin{enumerate}[(a)]
    \item $\tau=\eth_1$.
    \item $\tau(n)\leq \eth_k(n)$ for any $n,k \in \NN$.
    \item $\tau(a)\tau(b) \leq \eth_{2}(ab)$ for any $a,b \in \NN$.
    \item $\tau$ and $\eth_k$ have the same average order for any $k \in \NN$.
\end{enumerate}
Our first result will be the basis of all following upper bounds on divisor sums. It follows from the general
work \cite{BB06}, where they consider sums of suitable arithmetic functions over binary forms.

\begin{lemma}\label{plem:TBdivisor}
    Let $0 <X_1,X_2\leq X$ and let $F \in \ZZ[x_1,x_2]$ be a non-singular quartic binary form
    that is completely reducible over $\ZZ$. Then for any $n,k \in \NN$ and $\varepsilon >0$ we have
    $$\sum_{\substack{a\leq X_1 \\ b \leq X_2}} \eth^n_{k}(|F(a,b)|) \ll_{\varepsilon,k,n}
    ||F||^{\varepsilon}(X_1X_2(\log X)^{4(2^n-1)} + \max\{X_1,X_2\}^{1+\varepsilon}),$$
    where $||F||$ denotes the maximum absolute value of the coefficients of $F$.
\end{lemma}
\begin{proof}
    Let $F(x_1,x_2)=fx_1^{d_1}x_2^{d_2}G(x_1,x_2)$ where $f \in \ZZ$ and $G(x_1,x_2)$ is a primitive binary form
    with $G(1,0)G(0,1) \neq 0$. Then, \cite[Cor. 1]{BB06} implies that

    $$\sum_{\substack{a\leq X_1 \\ b \leq X_2}} \eth^n_{k}(|F(a,b)|)
    \ll_{\varepsilon,n,k} ||F||^{\varepsilon}\left(X_1X_2 E + \max\{X_1,X_2\}^{1 + \varepsilon}\right),$$
    where
    $$E=\prod_{p\leq \min\{X_1,X_2\}}
    \left(1 + \frac{\varrho^*_G(p)(\eth^n_k(p)-1)}{p}\right)\prod_{i=1,2}\prod_{p\leq
    X_i}\left(1 + \frac{d_i(\eth^n_k(p)-1)}{p}\right),$$
    and
    $$\varrho^*_G(m)=\frac{1}{\varphi(m)}
    \#\left\{
    \begin{array}{ll}
        (a,b) \in (0,m]^2:
    \end{array}
    \begin{array}{ll}
        &(a,b,m)=1 \\
        &G(a,b)\equiv 0\pmod m
    \end{array} \right\}.$$
    However, for any prime $p$ we have
    $$\varrho^*_G(p) \leq \frac{\# \{0 <a,b \leq p : G(a,b) \equiv 0 \pmod p\}}{p-1}
    \leq \frac{(4-d_1-d_2)p}{p-1},$$
    which implies that
    \begin{align*}
        E &\ll \prod_{p\leq X}\left(1 +
        \frac{4(2^n-1)}{p}\right) \ll(\log X)^{4(2^n-1)},
    \end{align*}
    as required.
\end{proof}

The next lemma handles the case of summing over more general regions than boxes.

\begin{lemma}\label{plem:divisor}
    Let $z_1,z_2,X>0$ and let $F \in \ZZ[x_1,x_2]$ be a non-singular quartic binary form
    that is completely reducible over $\ZZ$. Then for any $n,k \in \NN$ and $\varepsilon >0$ we have
    $$\sum_{\substack{a,b>0 \\ \max\{az_1,(b-a)z_2\} \leq b \leq X}} \eth^n_k(|F(a,b)|) \ll_{\varepsilon,n,k}
    \frac{||F||^{\varepsilon}X^2(\log X)^{4(2^n-1)}}{(z_1z_2)^{1-\varepsilon}}.$$
    Next suppose that $1 \leq y_1,y_2 \leq \log X$. Then for any $p,q,r\geq0$
    such that $p+q+r=2$ and $q>1$ we have
    \begin{align*}
    &\sum_{\substack{a,b \leq X \\ \max\{ay_1,(b-a)y_2\} < b}} \frac{\eth^n_k(|F(a,b)|)}{a^pb^q(b-a)^r}\ll_{\varepsilon,n,k}
    \frac{||F||^{\varepsilon}(\log X)^{4(2^n-1)+1}}{(y_1y_2)^{q-1}}. \\
    \end{align*}
\end{lemma}
\begin{proof}
    Throughout the proof, we suppress the dependence of the implied constant on $\varepsilon,n$ and $k$.
    In order to prove the first part of the lemma, we may assume that $z_1\leq X$ and $z_2 \leq X$, since otherwise the sum vanishes
    and the upper bound is clearly sufficient. We also emphasise that $a$ may be larger than $X$ in the case where $z_1<1$.
    We split the summation up into two cases, beginning with the case where $2a \leq b$. Here we may assume that $z_2 \leq 2$
    and hence $1 \ll z_2^{\varepsilon - 1}$, since again otherwise the sum will vanish.
    Summing over dyadic intervals and using Lemma \ref{plem:TBdivisor} gives
    \begin{align*}
        \sum_{\substack{\max\{az_1,2a\}\leq b \leq X }} \eth^n_k(|F(a,b)|)
        &\ll \frac{1}{z_2^{1-\varepsilon}}\sum_{\substack{\max\{Az_1,A\}<B \leq X\\ A,B \in 2^{\NN}}}
        \sum_{\substack{a \asymp A \\ b \asymp B}}\eth^n_k(|F(a,b)|)\\
        &\ll \frac{||F||^{\varepsilon}}{z_2^{1-\varepsilon}}
        \sum_{\substack{\max\{Az_1,A\}<B \leq X\\ A,B \in 2^{\NN}}}\left(AB(\log X)^{4(2^n-1)} + B^{1+\varepsilon}\right)\\
        &\ll \frac{||F||^{\varepsilon}(\log X)^{4(2^n-1)}}{z_2^{1-\varepsilon}}
        \left(\frac{X^2}{z_1} + X^{1+\varepsilon}\right)\\
        &\ll \frac{||F||^{\varepsilon}X^2(\log X)^{4(2^n-1)}}{(z_1z_2)^{1-\varepsilon}},
    \end{align*}
    since $z_1 \leq X$. For the case $b \leq 2a$, we again note that the sum vanishes unless $z_1 \leq 2$. Hence we have
    \begin{align*}
        \sum_{\substack{b \leq X\\ b \leq 2a \\ (b-a)z_2 \leq b}} \eth^n_k(|F(a,b)|)
        &\ll \sum_{\substack{ \max\{cz_2,2c\} \leq b \leq X}} \eth^n_k(|F(b-c,b)|) \\
        &\ll \frac{||F||^{\varepsilon}X^2(\log X)^{4(2^n-1)}}{(z_1z_2)^{1-\varepsilon}},\\
    \end{align*}
    as above.
    The proof of the second part of the lemma is very similar. When $2a \leq b$ we have
    \begin{align*}
        \sum_{\substack{a,b \leq X \\\max\{ay_1,2a\}\leq b \leq X }} \frac{\eth^n_k(|F(a,b)|)}{a^pb^q(b-a)^r}
        &\ll \frac{1}{y_2^{q-1}}
        \sum_{\substack{Ay_1<B \leq X\\A,B \in 2^{\NN}}}\frac{1}{A^{p+r}B^q}\sum_{\substack{a \asymp A \\ b\asymp B}}\eth^n_k(|F(a,b)|)\\
        &\ll \frac{||F||^{\varepsilon}}{y_2^{q-1}}
        \sum_{\substack{Ay_1<B \leq X\\A,B \in 2^{\NN}}}\left(A^{q-1}B^{1-q}(\log X)^{4(2^n-1)} + B^{1+\varepsilon-q}\right)\\
        &\ll \frac{||F||^{\varepsilon}(\log X)^{4(2^n-1)+1}}{(y_1y_2)^{q-1}},
    \end{align*}
    since $y_1^{q-1} \leq y_1 \leq \log X$.
    For the case $b \leq 2a$ we have
    \begin{align*}
        \sum_{\substack{a,b \leq X\\ a < b \leq 2a \\ (b-a)y_2 \leq b}} \frac{\eth^n_k(|F(a,b)|)}{a^pb^q(b-a)^r}
        &\ll \frac{1}{y_1^{q-1}}\sum_{\substack{c,b \leq X \\ \max\{cy_2,2c\}<b \leq X}} \frac{\eth^n_k(|F(b-c,b)|)}{b^qc^{p+r}} \\
        &\ll \frac{||F||^{\varepsilon}(\log X)^{4(2^n-1)+1}}{(y_1y_2)^{q-1}},
    \end{align*}
from above. This proves the lemma.
\end{proof}

The next result, while of a technical nature, will be used later on in our work.

\begin{lemma}\label{plem:technical}
    Let $X > 1$ and let $F \in \ZZ[x_1,x_2]$ be a non-singular quartic binary form
    that is completely reducible over $\ZZ$. Then for any $z_1,z_2 \geq 1$ and $\ee \in \NN^4$ we have
    $$\sum_{\substack{a , b \leq X \\ \max\{z_1a , z_2(b-a)\} \leq b \\ e_1|a, e_2|b \\ e_3|b+a, e_4 |b-a}}
    \eth_k(|F(a,b)|) \ll_{\varepsilon,k} \frac{||F||^{\varepsilon}X^2(\log X)^{4}}{
    \max\{e_1,e_2,e_3,e_4\}^{1-\varepsilon}(z_1z_2)^{1-\varepsilon}},$$
    for any $\varepsilon>0$.
\end{lemma}
\begin{proof}
    In what follows all implied constants are allowed to depend on $\varepsilon$ and $k$.
    As in the proof of Lemma \ref{plem:divisor}, we split the summation up into the cases $2a \leq b$ and $b \leq 2a$, and moreover we consider the cases where $\max\{e_1,e_2,e_3,e_4\}=e_i$ for some $i$. This leads to eight separate cases, which notate as follows.
    \begin{table}[hbt]
	\centering
	\begin{tabular}{lllll}
	$E_{11}$:& $e_1|a,$\quad &$2a\leq b$,\qquad \qquad & $E_{12}$: $e_1|a,$\quad &$b \leq 2a$,   \\
	$E_{21}$:& $e_2|b,$\quad &$2a\leq b$, \qquad \qquad& $E_{22}$: $e_2|b,$\quad &$b \leq 2a$,   \\
	$E_{31}$:& $e_3|b+a,$ \quad &$2a\leq b$, \qquad \qquad& $E_{32}$: $e_3|b+a,$ \quad &$b \leq 2a$, \\
	$E_{41}$:& $e_4|b-a,$ \quad &$2a\leq b$, \qquad \qquad& $E_{42}$: $e_4|b-a,$ \quad &$b \leq 2a$.  \\
	\end{tabular}
    \end{table}

    \noindent In the case $E_{11}$ we may assume that $z_2 \leq 2$, since otherwise the sum vanishes
    as in the proof of Lemma \ref{plem:divisor}. Here we have
    \begin{align*}
	\sum_{\substack{2a \leq b \leq X \\ z_1a \leq b \\ e_1|a}} \eth_k(|F(a,b)|)
 	&\ll \frac{1}{z_2^{1-\varepsilon}}\sum_{\substack{b \leq X \\ z_1 e_1 c \leq b }} \eth_k(|F(e_1c,b)|)
	\ll \frac{||F||^{\varepsilon}X^2(\log X)^4}{e_1^{1-\varepsilon}(z_1z_2)^{1-\varepsilon}},
    \end{align*}
    by Lemma \ref{plem:divisor}. The case $E_{42}$ can also be handled in a similar manner.
    For the case $E_{32}$ we have
    \begin{align*}
	\sum_{\substack{a , b \leq X \\ z_2(b-a) \leq b \leq 2a \\ e_3|b+a}} \eth_k(|F(a,b)|)
 	&\ll \frac{1}{z_1^{1-\varepsilon}}\sum_{\substack{d \leq X \\ z_2 c \leq d \\ e_3 | d}} \eth_k(|F(d-c,d+c)|)
	\ll \frac{||F||^{\varepsilon}X^2(\log X)^4}{e_3^{1-\varepsilon}(z_1z_2)^{1-\varepsilon}},
    \end{align*}
    again by Lemma \ref{plem:divisor}. This method also handles the cases $E_{21},E_{31},E_{12}$ and $E_{22}$,
    the key fact here being that the linear form which $e_i$ divides is bounded below by $b$.
    This leaves the last case $E_{41}$, where we have
    $$ \sum_{\substack{2a\leq b \leq X \\ z_1a \leq b \\ e_4|b-a}}\eth_k(|F(a,b)|)
        \ll\frac{1}{z_2^{1-\varepsilon}}\sum_{\substack{a\leq d \leq X \\ z_1a \leq d+a \\ e_4|d}}\eth_k(|F(a,d+a)|)
        \ll \frac{||F||^{\varepsilon}X^2(\log X)^4}{e_4^{1-\varepsilon}(z_1z_2)^{1-\varepsilon}}.$$
    Collecting these eight cases together completes the proof of the lemma.
\end{proof}

\section{Proof of Theorem \ref{thm:dp4asym}}
\subsection{The conic bundle structure}
As mentioned in the introduction, we begin the proof of Theorem \ref{thm:dp4asym}
by utilising the fact that $S$ has the structure of a conic bundle, at least away from the lines of $S$.
We have the following rational map
\begin{align*}
S &\dashrightarrow \PP^1 \\
x &\mapsto (x_0:x_2).
\end{align*}
The closure of the fibre over
a point $(a:b)$ with $ab\neq0$ is the rational curve
$$ax_2=bx_0 ,\quad bx_2=ax_1, \quad abx_3x_4=x_2^2(b^2-a^2),$$
on $S$. To proceed we choose a representative $(a,b)\in \ZZ^2$ of $(a:b) \in \Ptwo(\QQ)$ with
$(a,b)=1$ and $a>0$. Then we may pull back these rational curves to plane conics via the morphisms
\begin{equation}\label{def:psi_a,b}
\psi_{a,b}:\Ptwo \to \PP^4 , \quad
\psi_{a,b}: (x:y:z) \mapsto (a^2z:b^2z:abz: x : y),
\end{equation}
to get
$$N_U(B) = \sum_{\substack{(a,b) \in \ZZ^2\\ ab \neq 0, a>0 \\(a,b)=1}} N_{C_{a,b}}(B),$$
where
\begin{align}
    C_{a,b}:xy&=ab(b^2-a^2)z^2 \label{def:C_a,b}, \\
    N_{C_{a,b}}(B)&= \#\{ (x:y:z) \in C_{a,b}(\QQ) : H(\psi_{a,b}(x:y:z))\leq B, xyz\neq0\}. \nonumber
\end{align}
Note that here we are still using the height function $H$ given by the
embedding of $S$ into $\PP^4$.

\subsection{Reducing the range of summation}
The next simplification is to reduce the range of summation of $a$ and $b$, so that we may assume that they have roughly the same size.
\begin{lemma}\label{lem:reduce}
    We have
    \begin{align*}
        N_U(B)=4\sum_{(a,b) \in \regA^*} N_{C_{a,b}}(B)
             + O\left(\frac{B(\log B)^5}{(\log \log B)^{1/3}}\right),
    \end{align*}
    where we define
    $$\regA=\left\{(a,b) \in \RR^2:
    \begin{array}{ll}
        0<a< b \leq \sqrt{B},\\
        b-a > b/\log\log B,\\
        a > b/\log \log B.
    \end{array}\right\}, \quad \regA^*= \{ (a,b) \in \ZZ^2 \cap \regA : (a,b)=1\}.$$
\end{lemma}
\begin{proof}
    We begin by noting that
    $$N_{C_{a,b}}(B)=\frac{1}{2}\#\left\{
             (x,y,z) \in \ZZ^3:
            \begin{array}{ll}
                 (x,y,z)=1, xyz \neq0, \\
                 xy=z^2ab(b^2-a^2),\\
                  \max\{|x|,|y|,|a^2z|,|b^2z|\}\leq B.
            \end{array} \right\}.$$
    We now show that we may assume that $a < b$ by introducing a factor of $4$ into the counting problem.
    On noticing that the counting problem is invariant under the automorphism which negates $b$ and $x$,
    we see that we may assume that $b > 0$.
    Similarly we may assume that $b > a$, since the counting problem is again invariant
    under the automorphism which swaps $a$ and $b$ and negates $x$.
    Next, by Lemma \ref{plem:conic} the number of points on each conic is
    $$ N_{C_{a,b}}(B) \ll \tau(ab(b^2-a^2))\left(1 +\frac{B}{a^{1/3}b(b^2-a^2)^{1/3}}\right).$$
    However Lemma \ref{plem:divisor} implies that
    $$\sum_{a < b < \sqrt{B}}\tau(ab(b^2-a^2)) \ll B (\log
    B)^4.$$
    The contribution from  $ a\log \log B<b$ is
    $$B\sum_{\substack{a < b < \sqrt{B} \\ a\log \log B<b}}\frac{\tau(ab(b^2-a^2))}{a^{1/3}b(b^2-a^2)^{1/3}} \ll
    \frac{B(\log B)^5}{(\log \log B)^{1/3}},$$
    by Lemma \ref{plem:divisor}. Similarly, the contribution from $ (b-a)\log \log B < b$ is
    $$B\sum_{\substack{a < b < \sqrt{B} \\ (b-a)\log \log B < b}}\frac{\tau(ab(b^2-a^2))}{a^{1/3}b(b^2-a^2)^{1/3}} \ll
    \frac{B(\log B)^5}{(\log \log B)^{1/3}}.$$
\end{proof}

It is worth pointing out now that minor changes to the proof of Lemma \ref{lem:reduce} will yield the upper bound
$N_U(B) \ll B (\log B)^5$
for the counting problem. We will have to work significantly harder to get an asymptotic formula.

\subsection{Parameterising the conics}
In this section we count the number of points on each of the conics $C_{a,b}$, as given by (\ref{def:C_a,b}).
In what follows, we make frequent use of the fact that the coprimality of $a$ and $b$ implies that $(ab,b^2-a^2)=1$.
We may parameterise each of the conics via the morphisms
\begin{align*}
    \varphi_{a,b}:\Pone \to C_{a,b} \subset \Ptwo, \quad
    \varphi_{a,b}:(y_1:y_2) \mapsto (aby_1^2:(b^2-a^2)y_2^2:y_1y_2).
\end{align*}
Passing to the affine cone yields
$$N_{C_{a,b}}(B) = 2
\#\left\{\yy \in \NN^2:
    \begin{array}{ll} (y_1,y_2)=1, H(\psi_{a,b}(\varphi_{a,b}(\yy)))\leq B
    \end{array}
\right\},$$
where $\psi_{a,b}$ is given by (\ref{def:psi_a,b}).
To simplify notation we define
\begin{equation} \label{def:M(y)}
    M_{a,b}(\yy) = \max\{b^2y_1y_2,aby_1^2,(b^2-a^2)y_2^2\},
\end{equation}
to get
$$N_U(B) = 8N(B) + O\left(\frac{B(\log B)^5}{(\log \log B)^{1/3}}\right),$$
where
\begin{equation} \label{def:N(B)}
    N(B)= \sum_{\substack{(a,b) \in \regA^*}}
    \#\left\{\yy \in \NN^2:\begin{array}{ll} (y_1,y_2)=1,\\
    M_{a,b}(\yy)\leq (y_1,b^2-a^2) (y_2,ab)B\end{array}\right\}.
\end{equation}
Applying \Mob inversion, we find that
\begin{align*}
    N(B) &=  \sum_{(a,b) \in \regA^*}
    \sum_{\substack{\lambda_1|(b^2-a^2)\\\lambda_2|ab}}
    \#\left\{\yy \in \NN^2:\begin{array}{ll}
    (y_1,y_2)=1,  \lambda_i|y_i, \\
    (y_1/\lambda_1,(b^2-a^2)/\lambda_1)=1, \\
    (y_2/\lambda_2,ab/\lambda_2)=1,\\
    M_{a,b}(\yy)\leq \lambda_1\lambda_2B.
    \end{array}\right\} \\
    &= \sum_{\substack{(a,b) \in \regA^*
    }} \sum_{\substack{k_1\lambda_1|(b^2-a^2)\\ k_2\lambda_2|ab}}\mu(k_1k_2)
    \#\left\{\yy \in \NN^2:\begin{array}{ll}
    (y_1,y_2)=1,  k_i\lambda_i|y_i,\\
    M_{a,b}(\yy)\leq \lambda_1\lambda_2B.
    \end{array}\right\}. \\
\end{align*}
Our next step is to restrict the range of summation of the $\lambda_i$ and $k_i$, to make explicit the size
constraints implied by the expression $M_{a,b}(\yy)\leq \lambda_1\lambda_2B$.

\begin{lemma} \label{lem:k_i}
    For any $\varepsilon >0$ we have
    $$  N(B)= \sum_{(a,b) \in \regA^*}
        \sum_{\substack{\lambda_1k_1 |(b^2-a^2)\\ \lambda_2k_2|ab \\ k_1,k_2 \leq K}}\mu(k_1k_2)
        \#\left\{\yy \in \NN^2:\begin{array}{ll}
        (y_1,y_2)=1,  k_i\lambda_i|y_i\\
        M_{a,b}(\yy)\leq \lambda_1\lambda_2B
        \end{array}\right\}
        + O_\varepsilon\left(B(\log B)^{4+\varepsilon}\right), $$
    where $K=(\log B)^{1000}$ and the summation is subject to the condition
    \begin{equation}
        \frac{b^2K^2}{B} \leq \frac{k_1\lambda_1}{k_2\lambda_2} \leq \frac{B}{b^2K^2}. \label{div:restriction1}
    \end{equation}
\end{lemma}
\begin{proof}
    We first consider the contribution from $\max\{k_1,k_2\} \geq K$.
    In this case we may use Lemma \ref{plem:boxlattice} to count the number of $y_i$'s, to get an upper
    bound
    $$
     \sum_{(a,b) \in \regA^*} \sum_{\substack{\lambda_1k_1|(b^2-a^2)\\ \lambda_2k_2|ab \\ k_1k_2 \geq K}}
    \frac{B}{\sqrt{ab(b^2-a^2)}k_1k_2}
    \ll \frac{B (\log \log B)^2}{K}\sum_{a<b \leq \sqrt{B}} \frac{\tau^2(ab(b^2-a^2))}{b^2} \ll B,
    $$
    by Lemma \ref{plem:divisor}.
    We now show that we may restrict the range of summation to
    $k_1\lambda_1/k_2\lambda_2 \leq B/b^2K^2,$
    the lower bound being achieved in an analogous manner. Note that since
    $M_{a,b}(\yy)\leq \lambda_1\lambda_2B$ and $k_i\lambda_i |y_i$ for $i=1,2$,
    we deduce that we need only consider the contribution from
    \begin{equation} \label{eqn:kres1}
        \frac{B}{b^2K^2} \leq \frac{k_1\lambda_1}{k_2\lambda_2} \leq \frac{k_1^2\lambda_1}{\lambda_2} \leq \frac{B}{ab}.
    \end{equation}
    Using Lemma \ref{plem:boxlattice} again and summing over dyadic intervals we see that the contribution from (\ref{eqn:kres1}) is
    \begin{align*}
        \ll B(\log \log B)^2 \sum_{(a,b) \in \regA^*} \sum_{\substack{\lambda_1k_1 |(b^2-a^2)\\ \lambda_2k_2|ab \\ k_1,k_2
        \leq K \\ { (\ref{eqn:kres1})} \text{ holds}}} \frac{1}{k_1k_2b^2}
        \ll B(\log \log B)^4 \sum_{\substack{A \ll B \\ L_1,L_2 \ll A^2 \\ A,L_1,L_2 \in 2^{\NN}\\ { (\ref{eqn:kres2})} \text{ holds}}}\frac{1}{A^2} \sum_{\substack{(a,b) \in \regA^* \\ b \asymp A}}
        \sum_{\substack{\lambda_i \asymp L_i \\ \lambda_1 |(b^2-a^2)\\ \lambda_2|ab}} 1,
    \end{align*}
    where the sum is subject to the condition
    \begin{equation}
        \frac{B}{A^2K^3} \ll \frac{L_1}{L_2} \ll \frac{BK^2}{A^2} \label{eqn:kres2}.
    \end{equation}
    As in the Dirichlet hyperbola method, if $\lambda_2 \geq A$, say, then we may write $\mu_2=ab/\lambda_2$ and choose
    to sum over $\mu_2$ instead. Since $\lambda_2 \asymp L_2$ and $a \leq b$, we see that $\mu_2 \ll A^2/L_2 \ll A$ and
    $\mu_2 \gg A^2/(L_2 \log \log B)$. Using a similar trick with $\lambda_1$ gives
    \begin{equation}
        \ll B(\log \log B)^4 \sum_{\substack{A \ll B \\ L_1,L_2 \ll A^2\\A,L_1,L_2 \in 2^{\NN}\\ { (\ref{eqn:kres2})}
        \text{ holds}}}\frac{1}{A^2}
        \sum_{\substack{\lambda_i \ll f_2(L_i,A) \\ \lambda_i \gg f_1(L_i,A)}}
        \sum_{\substack{b \asymp A \\(a,b) \in \regA^* \\ \lambda_1 |(b^2-a^2) \\ \lambda_2|ab}} 1,        \label{eqn:k_i}
    \end{equation}
    where
    $$f_1(L_i,A)=\min\left\{L_i,\frac{A^2}{L_i (\log \log B)}\right\}, \quad
        f_2(L_i,A)=\min\left\{L_i,\frac{A^2}{L_i}\right\}.$$
    However we have
    \begin{align*}
        \sum_{\substack{b \asymp A \\(a,b) \in \regA^*\\ \lambda_1 |(b^2-a^2) \\ \lambda_2|ab}} 1
        &\ll \sum_{\substack{\lambda_1=e'e_1e_2 \\ \lambda_2=e_3e_4}}
        \sum_{\substack{(a,b) \in \Gamma_\ee \\ a,b \ll A \\ (a,b)=1 \\ e'=(\lambda_1,b+a,b-a)} }1
        \ll\sum_{\substack{\lambda_1=e'e_1e_2 \\ \lambda_2=e_3e_4 \\ (e_i,e_j)=1 \\e'|2, i\neq j}}
        \sum_{\substack{(a,b) \in \Gamma_\ee \\ a,b \ll A \\ (a,b)=1} }1,
    \end{align*}
    where we write $\Gamma_\ee = \{\xx \in \ZZ^2: e_1 | (b+a), e_2 | (b-a), e_3 | a, e_4 | b \}$.
    The coprimality of $e_i$ with $e_j$ for $i\neq j$ ensures that the lattice $\Gamma_\ee$ can be written
    as the intersection of four lattices with coprime determinants $e_1,e_2,e_3$ and $e_4$ respectively. Thus
    Lemma \ref{plem:boxlattice} implies that
    \begin{align*}
        \sum_{\substack{b \asymp A \\(a,b) \in \regA^* \\ \lambda_1 |(b^2-a^2) \\ \lambda_2|ab}} 1
        &\ll \sum_{\substack{\lambda_1=e'e_1e_2 \\ \lambda_2=e_3e_4 \\e'|2}}
        \left(\frac{A^2}{e_1e_2e_3e_4}+1\right)
        \ll \frac{\tau(\lambda_1)\tau(\lambda_2)A^2}{\lambda_{1}\lambda_{2}},
    \end{align*}
    since $\lambda_i \ll A$. Hence we find that (\ref{eqn:k_i}) is bounded above by
    \begin{align*}
        B(\log \log B)^4 \sum_{\substack{A \ll B \\ L_1,L_2 \ll A^2\\A,L_1,L_2 \in 2^{\NN}\\ { (\ref{eqn:kres2})}
        \text{ holds}}}
        \sum_{\substack{\lambda_i \ll f_2(L_i,A) \\ \lambda_i \gg f_1(L_i,A)}}
        \frac{\tau(\lambda_1)\tau(\lambda_2)}{\lambda_1\lambda_2}
        \ll B(\log B)^2(\log \log B)^4\sum_{\substack{A,L_1,L_2\ll B\\A,L_1,L_2 \in 2^{\NN}\\ { (\ref{eqn:kres2})}
        \text{ holds}}}1.
    \end{align*}
    The sum over those $L_1$ satisfying (\ref{eqn:kres2}) contributes $O(\log \log B)$,
    and the sum over $A$ and $L_2$ gives $O((\log B)^2)$, which is satisfactory for the lemma.
\end{proof}
We emphasise now that the condition (\ref{div:restriction1}) is very important to our work.
It is crucial for the handling of the error term in Lemma \ref{lem:f}, and it is this condition
which forced us to consider a \emph{restricted} divisor function in our work in Section \ref{sec:div}, rather than the usual
divisor function. Intriguingly, there is a purely geometrical interpretation for its appearance. We shall see in
the proof of Lemma \ref{lem:mainterm2} that it contributes towards the constant
$\alpha(\tS)$ appearing in the leading constant in Section \ref{sec:constant}.

We are now ready to handle the summation over $y_1$ and $y_2$.

\begin{lemma} \label{lem:f}
    For any $\varepsilon>0$ we have
    $$  N(B) = B\sum_{\substack{(a,b) \in \regA^* \\ \ell \leq B}}\frac{f(b/a)}{b^2}
        \sum_{\substack{\lambda_1k_1 |(b^2-a^2)\\ \lambda_2k_2|ab \\{ (\ref{div:restriction1})} \text{ holds}}}
        \frac{\mu(k_1k_2)\mu(\ell)(\ell,k_1k_2\lambda_1\lambda_2)}{\ell^2k_1k_2}
         + O_\varepsilon\left(B(\log B)^{4+\varepsilon}\right),$$
    where for $\theta >1$ we let
    $$f(\theta)=
    \vol\left\{\yy \in \RR_{>0}^2:\begin{array}{ll}
     y_1y_2 \leq 1, \\
    y_1^2 \leq \theta, y_2^2 \leq \theta^2/(\theta^2-1)
    \end{array}\right\}.$$
\end{lemma}
\begin{proof}
    Removing the coprimality conditions by \Mob inversion, the main term given by
    Lemma \ref{lem:k_i} has the form

    \begin{align*}
        &\sum_{\substack{(a,b) \in \regA^* \\ \ell \leq B}}
        \sum_{\substack{\lambda_1k_1 |(b^2-a^2)\\ \lambda_2k_2|ab \\ k_1,k_2 \leq K \\ { (\ref{div:restriction1})}
        \text{ holds}}}\mu(k_1k_2)\mu(\ell)
        \#\left\{\yy \in \NN^2:\begin{array}{ll}
             [\ell,k_i\lambda_i]|y_i, i=1,2,\\
            M_{a,b}(\yy)\leq \lambda_1\lambda_2B
            \end{array}\right\}.\\
    \end{align*}
    Letting $Y=\lambda_1\lambda_2B/b^2, q_i=[\ell,k_i\lambda_i],\theta=b/a$ and recalling the definition
    of $M_{a,b}(\yy)$ given in (\ref{def:M(y)}), we see that the number of $(y_1,y_2)$ is
    \begin{equation} \label{y:region}
        \#\left\{\yy \in \NN^2:\begin{array}{ll}
         q_i |y_i,i=1,2,\\
        y_1^2 \leq Y\theta, y_1y_2 \leq Y \\
        y_2^2 \leq Y\theta^2/(\theta^2-1)
        \end{array}\right\}.\\
    \end{equation}
    The first successive minimum of the lattice in (\ref{y:region}) is clearly $\min\{q_1,q_2\}$, and one can check that the boundary of the region
    in question has length $\ll \sqrt{Y\theta} +\sqrt{Y(\theta^2/(\theta^2-1)}$. It follows from
    Lemma \ref{plem:lattice} that (\ref{y:region}) equals
    $$\frac{Yf(b/a)}{q_1q_2} +
    O\left(\frac{\log \log B\sqrt{\lambda_1\lambda_2
    B}}{b\min\{q_1,q_2\}} \right).$$
    In order to handle the error term, we only consider the
    case $[\ell,k_1\lambda_1]\leq[\ell,k_2\lambda_2]$, the other case being
    dealt with in almost exactly the same manner. The error term here contributes
    \begin{align*}
        &\sqrt{B}\log \log B\sum_{\substack{(a,b) \in \regA^* \\ \ell \leq B}}
        \sum_{\substack{\lambda_1k_1 |(b^2-a^2)\\ \lambda_2k_2|ab \\ k_1,k_2 \leq K \\{ (\ref{div:restriction1})} \text{ holds}}}
        \frac{(\ell,k_1\lambda_1)\sqrt{\lambda_2}}{b\ell k_1 \sqrt{\lambda_1}} \\
        \ll &\sqrt{B}\log B\sum_{\substack{(a,b) \in \regA^* \\ \ell \leq B}}
        \sum_{\substack{\lambda_1k_1 |(b^2-a^2)\\ \lambda_2k_2|ab \\ { (\ref{div:restriction1})} \text{ holds}}}
        \sum_{d|(\ell,k_1\lambda_1)}
        \frac{d\sqrt{k_2\lambda_2}}{b\ell \sqrt{k_1\lambda_1}} \\
        \ll &\frac{B(\log B)^2}{K}\sum_{(a,b) \in \regA^*}
        \sum_{\substack{\lambda_1k_1 |(b^2-a^2)\\ \lambda_2k_2|ab}}
        \frac{\tau(k_1\lambda_1)}{b^2},
    \end{align*}
    by (\ref{div:restriction1}). Moreover, it is clear on applying Lemma \ref{plem:divisor} that this is bounded above by $O(B)$,
    since we chose $K$ in Lemma \ref{lem:k_i} to be a very large power of a logarithm.
    We finish the proof by showing that we may extend the sum over the $k_i$ to infinity.
    We note that by Lemma \ref{lem:reduce} we have the upper bound
    \begin{equation} \label{f:upperbound}
        f(b/a)\leq \frac{\sqrt{b}}{\sqrt{a}} \cdot\frac{b}{\sqrt{b^2-a^2}} \leq \log \log B.
    \end{equation}
    Hence by Lemma \ref{plem:divisor}, the contribution to the main term from $\max\{k_1,k_2\} \geq K$ is
    \begin{align*}
        &\ll \frac{B\log \log B}{K}\sum_{\substack{(a,b) \in \regA^* \\ \ell \leq B}}\frac{1}{b^2}
            \sum_{\substack{\lambda_1k_1 |(b^2-a^2)\\ \lambda_2k_2|ab }}
            \frac{(\ell,k_1k_2\lambda_1\lambda_2)}{\ell^2} \\
        &\ll \frac{B(\log B)^2}{K}\sum_{\substack{a<b\leq \sqrt{B}}}\frac{\tau^2(ab(b^2-a^2))}{b^2}
        \ll B,
    \end{align*}
    which is satisfactory.
\end{proof}

\subsection{The restricted divisor problem}
It now remains to deal with the main term of $N_U(B)$, which by Lemma \ref{lem:f} has the form
\begin{equation} \label{eq:mainterm}
    8B\sum_{\substack{(a,b) \in \regA^* \\ \ell \leq B}}\frac{f(b/a)}{b^2}
    \sum_{\substack{\lambda_1k_1 |(b^2-a^2)\\ \lambda_2k_2|ab \\ { (\ref{div:restriction1})} \text{ holds}}}
    \frac{\mu(k_1k_2)\mu(\ell)(\ell,k_1k_2\lambda_1\lambda_2)}{\ell^2k_1k_2},
\end{equation}
where $f$ is as given in Lemma \ref{lem:f} and $\regA^*$ is as in Lemma \ref{lem:reduce}. Our aim is to get this into the form
of a restricted divisor sum, so that we may use the work in Section \ref{sec:div}. Before we do this however, we need to introduce some notation.
Define a multiplicative function $h$ by
\begin{equation}\label{def:h}
    h(p^a)= \frac{2\mu(p^a)}{p+1},
\end{equation}
for any prime $p$ and $a \in \NN$. We then define linear forms
\begin{equation} \label{eqn:ell_i}
    \ell_1(a,b)=a, \quad \ell_2(a,b)=b, \quad \ell_3(a,b)=b+a, \quad \ell_4(a,b)=b-a.
\end{equation}
As in Secion \ref{sec:div}, we shall also be interested in the lattice $\Gamma_\dd$, defined for any $\dd \in \NN^4$ by
\begin{equation} \label{eqn:Gammadp4}
    \Gamma_\dd = \left\{\xx \in \ZZ^2: d_i | \ell_i(\xx), (i=1,2,3,4) \right\}.
\end{equation}

\begin{lemma}
 \label{lem:mainterm}
    We have
    $$  N_U(B)=\frac{8B}{\zeta(2)}\sum_{\substack{\ee \in \NN^5 \\ v \in \NN}}h(er)\mu(v)
        \sum_{\substack{r,s|2 \\ (er,s)=e_0}}\mu(r)\mu(s)F(\ee,r,s,v,B)
        + O\left(\frac{B (\log B)^5 }{(\log \log B)^{1/3}}\right),$$
    where we write $\ee=(e_0,e_1,e_2,e_3,e_4)$ and $e=e_0e_1e_2e_3e_4$. Here
    $$F(\ee,r,s,v,B)=\sum_{\substack{(a,b) \in \Gamma_\mm \cap \regA}}
        \frac{f(b/a)}{b^2}\sum_{\substack{i \in \{1,2\} \\ e_id_i | \ell_i(a,b)  }}
        \sum_{\substack{j \in \{3,4\} \\rse_jd_j| \ell_j(a,b) \\ { (\ref{div:restriction2})} \text{ holds}}}1,$$
    where we let
    $$\mm=([e_1,v],[e_2,v],rse_3,rse_4),$$
    and the sum is subject to the condition
    \begin{equation}
    \frac{b^2K^2}{B} \leq \frac{e_1e_2d_1d_2}{r^2se_3e_4d_3d_4} \leq
    \frac{B}{b^2K^2}. \label{div:restriction2}
    \end{equation}
\end{lemma}
\begin{proof}
    We first simplify (\ref{eq:mainterm}) by performing the summation over $\ell$.
    This is achieved by noting that
    \begin{align*}
        \sum_{\ell=1}^\infty\frac{\mu(\ell)(\ell,k_1k_2\lambda_1\lambda_2)}{\ell^2}
        &= \prod_p \left(1 - \frac{(p,k_1k_2\lambda_1\lambda_2)}{p^2}\right)
        = \frac{1}{\zeta(2)\varphi^{\dag}(k_1k_2\lambda_1\lambda_2)},
    \end{align*}
    where $$\varphi^{\dag}(n)=\prod_{p|n}\left(1 + \frac{1}{p}\right).$$
    By (\ref{f:upperbound}) the contribution from $\ell \geq B$ is
    \begin{align*}
        &\ll_\varepsilon B^{1+\varepsilon}\sum_{\substack{(a,b) \in \regA^* \\ \ell \geq B}}\frac{1}{b^2}
        \sum_{\substack{\lambda_1k_1 |(b^2-a^2)\\ \lambda_2k_2|ab }}
        \frac{(\ell,k_1k_2\lambda_1\lambda_2)}{\ell^2k_1k_2}
        \ll_\varepsilon B^{\varepsilon}\sum_{\substack{a<b\leq B}}\frac{1}{b^2}
        \ll_\varepsilon B^{\varepsilon}(\log B),
    \end{align*}
    for any $\varepsilon >0$.
    So on referring to (\ref{eq:mainterm}), we see that we may write
    $$N_U(B)=\frac{8B}{\zeta(2)}\sum_{(a,b) \in \regA^*} \frac{f(b/a)\Theta(a,b)}{b^2}+
    O\left(\frac{B (\log B)^5 }{(\log \log B)^{1/3}}\right),$$
    where
    $$\Theta(a,b)= \sum_{\substack{\lambda_1k_1 |(b^2-a^2)\\ \lambda_2k_2|ab \\ { (\ref{div:restriction1})} \text{ holds}}}
        \frac{\mu(k_1k_2)}{\varphi^{\dag}(k_1k_2\lambda_1\lambda_2)k_1k_2}.$$
    If $I(d;X)$ denotes the characteristic function of the set
    $\left\{ d \in \RR_{>0}: 1/X \leq d \leq X\right\},$ then we have
    \begin{align*}
        \Theta(a,b)&=\sum_{\substack{d_1 |ab\\ d_2| (b^2-a^2)}} \frac{I\left(\frac{d_1}{d_2};\frac{B}{b^2K^2}\right)} {
        \varphi^{\dag}(d_1d_2)}\sum_{k_i|d_i}\frac{\mu(k_1k_2)}{k_1k_2} \\
        &=\sum_{\substack{d_1 |ab\\ d_2| (b^2-a^2)}} I\left(\frac{d_1}{d_2};\frac{B}{b^2K^2}\right)
        \sum_{e|d_1d_2}h(e),
    \end{align*}
    where $h$ is given by (\ref{def:h}).
    Also note that for any arithmetic function $g$ we have
    \begin{equation}
        \sum_{d|n_1n_2}g(d) \label{eqn:g}
        =\sum_{\substack{k|n_1,n_2}}\mu(k)\sum_{\substack{kd_i | n_i}}g(kd_1d_2).
    \end{equation}
    Using this we find that
    \begin{align*}
        \Theta(a,b) &=\sum_{\substack{d_i|\ell_i(a,b)}}
        \sum_{\substack{sd_3 | b+a \\ sd_4 | b-a}}\mu(s)\sum_{e|sd_1d_2d_3d_4}h(e)
        I\left(\frac{d_1d_2}{sd_3d_4};\frac{B}{b^2K^2}\right),
    \end{align*}
    where the $\ell_i$ are given by (\ref{eqn:ell_i}).
    Using (\ref{eqn:g}) again we have
    \begin{align*}
        \Theta(a,b) &=\sum_{\substack{e \in \NN \\d_i|\ell_i(a,b)}}\sum_{\substack{sd_3 | b+a \\ sd_4 | b-a}}\mu(s)
        \sum_{\substack{e=e_0e_1e_2e'  \\ e_1|d_1, e_2|d_2 \\ e'|d_3d_4 \\ (e,s)=e_0}}h(e)
        I\left(\frac{d_1d_2}{sd_3d_4};\frac{B}{b^2K^2}\right)\\
        &=\sum_{\substack{e \in \NN \\d_i|\ell_i(a,b)}}\sum_{\substack{sd_3 | b+a \\ sd_4 | b-a}}\mu(s)
        \sum_{\substack{e=e_0e_1e_2e_3e_4 \\ e_1|d_1, e_2|d_2 \\ re_3|d_3, re_4|d_4 \\ (er,s)=e_0 }}
        h(er) \mu(r)
        I\left(\frac{d_1d_2}{sd_3d_4};\frac{B}{b^2K^2}\right).
    \end{align*}
    We now make the change of variables
    $$ d_1 \mapsto e_1d_1, \quad d_2 \mapsto e_2d_2, \quad d_3 \mapsto re_3d_3, \quad d_4 \mapsto re_4d_4,$$
    which allows us to move the summation over $e,r,s$ to the outside, as in the statement of the lemma. Note
    that $r,s | 2$ since $(a,b)=1$ implies that $(b+a,b-a) | 2$. The proof of the lemma is then complete on removing
    the coprimality condition on $a$ and $b$.
\end{proof}
The main term of $N_U(B)$ is now written so that it visibly involves a restricted divisor sum, which we may handle using Corollary \ref{cor:resdivisor}.

\begin{lemma} \label{lem:mainterm2}
    We have
    $$N_U(B)=\alpha(\tS) \tau_\infty B(\log B)^5\prod_p\left(1-\frac{1}{p}\right)^5\left(1+\frac{1}{p}\right)\sigma_p  \left(1 + o(1)\right),$$
    where for every prime $p$ we let
    \begin{align*}
        \sigma_p=
        \sum_{\substack{\eepsilon \in \{0,1\}^5 \\ 0\leq \nu \leq 1 }}
        \sum_{\substack{0 \leq \varrho,\sigma \leq \nu_2(p) \\ 0 \leq \epsilon - \epsilon_0 + \varrho + \sigma \leq 1\\
        0 \leq \epsilon_0 \leq \sigma}}
        \sum_{\substack{\kk \in \ZZ^4_{\geq0}}}
        \frac{(-1)^{\nu+\varrho+\sigma}h(p^{\epsilon + \varrho})}
        {\rho_0(p^{\max\{\nu,N_1\}},p^{\max\{\nu,N_2\}},p^{N_3},p^{N_4})},
    \end{align*}
    where we write $\eepsilon=(\epsilon_0,\epsilon_1,\epsilon_2,\epsilon_3,\epsilon_4)$ and
    $\epsilon=\sum_{i=0}^4\epsilon_i$.
    Here $\nu_2$ denotes the $2$-adic valuation, $h$ is given by (\ref{def:h}) and $\rho_0(\dd)=\det \Gamma_\dd,$
    where $\Gamma_\dd$ is given by (\ref{eqn:Gammadp4}). Also
    $$\begin{array}{ll}
        N_i = \epsilon_i + k_i, &i \in \{1,2\}, \\
        N_j = \varrho + \sigma + \epsilon_j + k_j,  &j \in \{3,4\},
    \end{array}$$
    and $\alpha(\tS)$ and $\tau_\infty$ are the factors appearing in the leading constant of Manin's conjecture as described in Section \ref{sec:constant}.
\end{lemma}
\begin{proof}
    We begin by letting
    $$\regA(\yy) = \left\{ (a,b) \in \regA : \frac{ay_1^2}{b} \leq 1, \left(1- \frac{a^2}{b^2}\right)y_2^2 \leq 1\right\}.$$
    Then, recalling the definition of $f$ given in Lemma \ref{lem:f}, and using the same notation as Lemma \ref{lem:mainterm}, we see that we have
    $$
        F(\ee,r,s,v,B)= \int_{\substack{y_1y_2 \leq 1,\\0 \leq y_i \leq \log \log B}}
        \sum_{\substack{(a,b) \in \Gamma_\mm\cap \regA(\yy)}}
        \frac{1}{b^2}\sum_{\substack{i \in \{1,2\} \\ e_id_i | \ell_i(a,b)  }} \sum_{\substack{j \in \{3,4\} \\rse_jd_j| \ell_j(a,b)
         \\ { (\ref{div:restriction2})} \text{ holds}}}d \yy.
    $$
    We can now apply Corollary \ref{cor:resdivisor} where we take $X = \sqrt{B},\regR=\regA(\yy)/X, \Lambda=\Gamma_\mm$ and
    $V'=V'(e,r,s)$ to be the set corresponding to (\ref{div:restriction2}). This gives
    $$ F(\ee,r,s,v,B)=\frac{\prod_p C_p(\mm)(\log B)^5}{2^4\det \Gamma_\mm}
    \int_{\substack{y_1y_2 \leq 1,\\ 0 \leq y_i \leq \log \log B}}C'_\infty \mathrm d \yy + O_{\ee,r,s,v}((\log B)^{4+\varepsilon}),
    $$
    where
    $$  C'_\infty = \vol \regR  \int_{\substack{u \in [0,1] \\ \eeta \in [1,u]^4}} \chi_{V'}(\eeta,u) \mathrm d \eeta \mathrm d u, \quad
        C_p(\mm)=\left(1-\frac{1}{p}\right)^4\sum_{\kk \in \ZZ^4_{\geq0}}
        \frac{\det \Gamma_\mm}{\det \Gamma_\mm(p^{k_1},p^{k_2},p^{k_3},p^{k_4})}.
    $$
    We begin by simplifying the non-archimedean factor. Here, unraveling definitions we find that
    $$\Gamma_\mm(\dd)= \left\{\xx \in \Gamma_\mm:
        \begin{array}{ll}
           d_i | \ell_i(\xx)/e_i, & i=1,2,\\
           d_j |\ell_j(\xx)/rse_j, & j=3,4.
        \end{array}\right\}.$$
    Recalling the definition of $\mm$ in Lemma \ref{lem:mainterm} and noticing that
    $[[e,v],ed]=[v,ed]$ for any $e,v,d \in \NN$, we see that
    $$
        \Gamma_\mm(\dd)=
        \left\{\xx \in \ZZ^2:
        \begin{array}{ll}
           [v,e_id_i] | \ell_i(\xx), & i=1,2,\\
           rse_jd_j |\ell_j(\xx), & j=3,4.
        \end{array}\right\}.$$
    Thus we deduce that $C_p(\mm)(1-1/p^2)=\sigma_p$, on taking the Euler product
    of the sum over $\ee,r,s,v$ in Lemma \ref{lem:mainterm}.
    For the archimedean factor we have
    \begin{equation} \label{eqn:yintegral}
        \int_{\substack{y_1y_2 \leq 1\\ 0 \leq y_1,y_2 \leq \log \log B}}\vol{\regR}\mathrm d \yy
        =\int_{\substack{y_1y_2 \leq 1 \\ y_1,y_2 \geq 0}}
        \int_{\substack{0<a<b<1 \\ ay_1^2\leq b \leq a\log \log B \\
        \left(1- a^2/b^2\right)y_2^2 \leq 1 \\ b \leq (b-a)\log \log B}} \intd a \intd b \mathrm d \yy. \\
    \end{equation}
    Performing the integration over $\yy$, we see that the contribution from $b \geq a\log \log B$ is bounded above by
    \begin{align*}
    \int_{\substack{0<a<b<1 \\ a\log \log B \leq b }}\frac{b^{3/2}}{(a(b^2-a^2))^{1/2}} \intd a \intd b
    \ll \frac{1}{(\log \log B)^{1/2}}.
    \end{align*}
    While the contribution from $b \geq (b-a)\log \log B$ is handled in a similar manner. Hence making
    the change variables $y_0=a/b$ and evaluating the integral over $b$, we see that (\ref{eqn:yintegral})
    is equal to
    \begin{align*}
        &\frac{1}{2}\int_{\substack{ y_1y_2 \leq 1 \\ y_1,y_2 \geq 0}}
        \int_{\substack{0<y_0<1 \\ y_0y_1^2 \leq 1 \\\left(1- y_0^2\right)y_2^2 \leq 1}} \mathrm d \yy
        + O\left(\frac{1}{(\log \log B)^{1/2}}\right).\\
    \end{align*}
    We now use the change of variables 
    $$x_0=y_0^2y_1y_2, \quad x_1=y_1y_2, \quad x_3=y_0y_1^2,$$
    to see that (\ref{eqn:yintegral}) equals
    \begin{align*}
        &\frac{1}{2}\int_{\{\xx \in \RR^3 : 0< x_0/x_1, x_1, x_3,x_0x_1(x_1-x_0)^2/x_3^2 \leq 1 \}} \frac{\intd x_0 \intd x_1 \intd x_3}{4(x_0x_1)^{1/2}x_3}
        = \frac{\tau_\infty}{32}.\\
    \end{align*}
    For the alpha constant, note that we have
    \begin{align*}
        \int_{\substack{u \in [0,1] \\ \eeta \in [1,u]^4}} \chi_{V'}(\eeta,u) \mathrm d \eeta \mathrm d u
        &= \int_{\substack{u \in [0,1] \\ \eeta \in [1,u]^4}} \chi_{V''}(\eeta,u) \mathrm d \eeta \mathrm d u
        +O_{\ee,r,s}\left(\frac{1}{\log \log B}\right),
    \end{align*}
    where now
    $$V''=\{(\eeta,u) \in [0,1]^5:2u -2 \leq \eta_1 + \eta_2 - \eta_3 - \eta_4 \leq 2 - 2u\}.$$
    We are thus lead to calculate the volume of some rational polytope. One can use \cite{Fra09},
    for example, to find that
    $$\int_{\substack{u \in [0,1] \\ \eeta \in [1,u]^4}}\chi_{V''}(\eeta,u) \mathrm d \eeta \mathrm d u = \frac{4}{45}=64\alpha(\tS).$$
    It thus remains to show that we may control our non-uniform error
    when we sum over $e,r,s$ and $v$. To do this, we use an argument based on
    the dominated convergence theorem, reminiscent of Heath-Brown \cite[Lem. 6.1]{HB03}. Let
    $$\mathcal{E}(\ee,r,s,v;B)= \frac{F(\ee,r,s,v,B)}{(\log B)^5} - \frac{\alpha(\tS) \tau_\infty  \prod_p C_p(\mm)}{8\det \Gamma_\mm}.$$
    For fixed $\ee,r,s$ and $v$ we have shown that $\mathcal{E}(\ee,r,s,v;B) \to 0$ as $B \to \infty$. So in order
    to finish the proof the lemma, we need to show the dominated convergence of the sum
    \begin{equation}
        \sum_{\substack{\ee \in \NN^5\\ v \in \NN}}
        \sum_{\substack{r,s|2 \\ (er,s)=e_0 }}
        |h(er)\mu(r)\mu(s)\mu(v)\mathcal{E}(\ee,r,s,v;B)|, \label{eqn:dominatedconvergence}
    \end{equation}
    where as before we write $e=e_0e_1e_2e_3e_4$.
    I claim that it is sufficient to give the upper bound
    $$ \mathcal{E}(\ee,r,s,v;B) \ll_\varepsilon \frac{1}{e^{\varepsilon}v^{1+\varepsilon}}, $$
    for any $\varepsilon >0$. Indeed, in this case (\ref{eqn:dominatedconvergence}) is bounded above by
    \begin{align*}
        \sum_{\substack{\ee \in \NN^5\\ v \in \NN}}\frac{|h(e)\mu(v)|}{e^{\varepsilon}v^{1+\varepsilon}}
        \ll_\varepsilon \sum_{\substack{\ee \in \NN^5\\ v \in \NN}}\frac{1}{e^{1+\varepsilon}v^{1+\varepsilon}} \ll_\varepsilon 1,
    \end{align*}
    since we have $h(e)\ll 1/e$ by definition. We note that we have
    $$\mathcal{E}(\ee,r,s,v;B) \ll
    \frac{1}{(\log B)^{5}}\sum_{\substack{(a,b) \in \Gamma_\mm \\ a < b \leq \sqrt{B}}}
        \frac{f(b/a)\eth_4(ab(b^2-a^2))}{b^2}
        + \frac{\prod_p |C_p(\mm)|}{\det \Gamma_\mm},$$
    where $\eth_4$ is given by (\ref{def:eth}). The upper bound $\prod_p |C_p(\mm)| \ll (ev)^{\varepsilon}$
    follows from Theorem \ref{thm:resdivisor}. By Lemma \ref{lem:rho}, we know that $\det \Gamma_\mm\gg [e,v^2]$,
    since $(e_i,e_j)=1$ for all $i \neq j$ as $e$ is square-free. On the other hand,
    we have
    $$\sum_{\substack{(a,b) \in \Gamma_\mm \\ a < b \leq \sqrt{B}}} \frac{f(b/a)\eth_4(ab(b^2-a^2))}{b^2}
    \ll \int_{y_1,y_2>0}\int_1^{\sqrt{B}} \frac{1}{t^3}
    \sum_{\substack{(a,b) \in \Gamma_\mm \\a <b \leq t  \\ \max\{y_1^2a,(b-a)y_2^2\} \leq b}}\eth_4\left(ab(b^2-a^2)\right)\intd t \intd \yy.
    $$
    Thus the result follows after making the change of variables $a=a'v,b=b'v$, and applying Lemma \ref{plem:technical}
    to deduce that for any $t>1$ we have
    \begin{align*}
        \sum_{\substack{(a,b) \in \Gamma_\mm \\a <b \leq t \\ \max\{y_1^2a,(b-a)y_2^2\} \leq b}}\eth_4\left(ab(b^2-a^2)\right)
        \ll_\varepsilon& \frac{t^2(\log t)^4}{\max\{1,y_1^2\}\max\{1,y_2^2\}||\ee||^{1-\varepsilon}v^{2-\varepsilon}}.
    \end{align*}
\end{proof}

\subsection{The local densities}

To complete the proof of Theorem \ref{thm:dp4asym}, it remains to show that for any prime $p$ we have
\begin{equation}
    \left(1-\frac{1}{p}\right)^5\left(1+\frac{1}{p}\right)\sigma_p = \tau_p, \label{eqn:showpadic}
\end{equation}
where $\tau_p$ is given in Section \ref{sec:constant} and $\sigma_p$ in Lemma \ref{lem:mainterm2}.
In order to do this, we need to have an explicit expression for the function $\rho_0$ defined in Lemma \ref{lem:mainterm2}.

\begin{lemma} \label{lem:rho0}
    Let $p$ be a prime and let $\ee \in \ZZ^4_{\geq0}$. If $p=2$ and $\min\{e_3,e_4\}>\max\{e_1,e_2\}$ then
    $$\rho_0(2^{e_1},2^{e_2},2^{e_3},2^{e_4}) = 2^{e_3 + e_4 -1}.$$
    Otherwise
    $$\rho_0(p^{e_1},p^{e_2},p^{e_3},p^{e_4}) =    p^{e_{\sigma(1)}+e_{\sigma(2)}},$$
    where we have chosen a permutation $\sigma$ such that $e_{\sigma(1)} \geq
    e_{\sigma(2)}\geq e_{\sigma(3)} \geq e_{\sigma(4)}.$
\end{lemma}
\begin{proof}
    By Lemma \ref{lem:rho}, we see that we need only consider the case $p=2$. Moreover
    the same method given there works if $\min\{e_3,e_4\}\leq\max\{e_1,e_2\}$, thus we may assume that
    $\min\{e_3,e_4\}>\max\{e_1,e_2\}$.
    When $e_3 \geq e_4$, it is sufficient to show that $2^{e_3}|(b + a)$ and $2^{e_4}|(b-a)$ if and only if
    $2^{e_3}|(b+a), 2^{e_4-1}|b$ and $2^{e_4-1}|a$. Indeed, this lattice has determinant $2^{e_3+e_4-1}$.

    For the first implication, we have $2^{e_3}|(b + a)$ and $2^{e_4}|(b-a)$ clearly implies that $2^{e_4}|2b$ and $2^{e_4}|2a$,
    as required.
    For the other implication, assume that $2^{e_3}|(b+a), 2^{e_4-1}|b,2^{e_4-1}|a$ and write
    $a=2^{e_4-1}a'$ and $b=2^{e_4-1}b'$. Then $2^{e_3 -e_4 +1}|(b'+a')$, and hence $a'$ and
    $b'$ share the same parity so $2|(b'-a')$. Hence $2^{e_4}|(b-a)$ as required.
    The proof in the case $e_4 \geq e_3$ works in a similar manner.
\end{proof}

Now let $p$ be any prime. In order to show (\ref{eqn:showpadic}), we split the summation
over the $N_i$ (in the notation of Lemma \ref{lem:mainterm2}) into various cases. First, the contribution
from the case where $N_i=0$ for all $i=1,2,3,4$ is
\begin{align*}
    &\sum_{\substack{0\leq \nu \leq 1}}\frac{(-1)^{\nu}}{\rho_0(p^\nu,p^\nu,1,1)}
    = 1 - \frac{1}{p^2}.
\end{align*}
Next, we handle the case where $N_i \geq 1$ for some $i$ and $N_j=0$ for all $i\neq j$.
Note that since $N_3=0$ or $N_4=0$ we must have $\varrho=\sigma=\epsilon_0=0$. So we get
\begin{align*}
    \sum_{\substack{\epsilon+k \geq 1 \\ 0\leq \nu \leq 1}}
    \frac{(-1)^{\nu}h(p^{\epsilon})}{p^{\epsilon+k+\nu}}
    =&\left(1 - \frac{1}{p}\right)\sum_{\substack{\epsilon+k \geq 1}}
    \frac{h(p^{\epsilon})}{p^{\epsilon+k}} \\
    =&\left(1 - \frac{1}{p}\right)\left(\frac{h(p)}{p} +
    \sum_{\substack{k \geq 1 \\ 0\leq \epsilon \leq 1}}\frac{h(p^{\epsilon})}{p^{\epsilon+k}}\right)\\
    =&\left(1 - \frac{1}{p}\right)\frac{h(p)}{p} + \frac{1}{p}\left(1 + \frac{h(p)}{p}\right)\\
    =&\frac{1 + h(p)}{p}
    =\left(1 - \frac{1}{p}\right)\frac{1}{p+1},
\end{align*}
since we have $h(p)=-2/(p+1)$ by definition (\ref{def:h}).
Hence, the total contribution from these cases is
\begin{align*}
    &\left(1 - \frac{1}{p}\right)\left(1 + \frac{1}{p} + \frac{4}{p+1}\right)
    = \frac{(1-1/p)(1 + 6/p + 1/p^2)}{1 + 1/p}.
\end{align*}
Recalling the definition of $\tau_p$ in Section \ref{sec:constant}, in order to prove (\ref{eqn:showpadic})
it suffices to show that if $N_i \geq 1$ and $N_j \geq 1$
for some $i\neq j$, then the sum given in Lemma \ref{lem:mainterm2} vanishes.

If $p\neq 2$, then in this case Lemma \ref{lem:rho0} implies that the function $\rho_0$ is independent of $\nu$, and
changing the order of summation we have $\sum_{0 \leq \nu \leq 1}(-1)^\nu = 0$. This
is simply a reflection of the fact that in the original counting problem,
we were only counting those $a$ and $b$ which were coprime.
For the case $p=2$, a similar argument shows that the sum vanishes if $N_i,N_j\geq 1$ for some $(i,j) \neq (3,4),(4,3)$,
or $N_3,N_4 \geq 2.$
Therefore we need to consider the extra cases given by $N_1=N_2=0,N_3=1, N_4 \geq 1$ and $N_1=N_2=0,N_4=1, N_3 \geq 1$.
For any $N\in \NN$ we have
\begin{align*}
    &\sum_{\substack{0\leq \nu \leq 1 \\ k_3,k_4 \geq 0}}
    \sum_{\substack{0 \leq \epsilon_3 + \epsilon_4 + \varrho + \sigma  \leq 1 \\ 0 \leq \epsilon_0 \leq \sigma \\N_3=1, N_4=N}}
    \frac{(-1)^{\nu+\varrho+\sigma}h(2^{\epsilon_0 + \epsilon_3+\epsilon_4 + \varrho})}{\rho_0(2^\nu,2^\nu,2,2^{N_4})}\\
    &=\frac{1}{2^N}\left(1 - \frac{1}{2}\right)
    \sum_{\substack{k_4 \geq 0}}
    \sum_{\substack{0 \leq \epsilon_3 + \epsilon_4 + \sigma + \varrho\leq 1 \\ 0 \leq \epsilon_0 \leq \sigma,  k_3 \geq 0 \\N_3=1, N_4=N}}
    (-1)^{\varrho+\sigma}h(2^{\epsilon_0 + \epsilon_3+\epsilon_4 + \varrho}).
\end{align*}
However, this inner sum vanishes.
Indeed, the condition $N_3=1$ implies that only one of $\varrho,\sigma, \epsilon_3$ and $k_3$ may be non-zero.
The contribution from each case is  $-h(2),-1 - h(2), h(2)$ and $1 + h(2)$, respectively.
The obvious symmetry means we that the sum also vanishes for $N_3=N$ and $N_4=1$. Thus we have shown (\ref{eqn:showpadic}),
and combining this with Lemma \ref{lem:mainterm2} completes the proof of Theorem \ref{thm:dp4asym}.




\begin{thebibliography}{xx}

\bibitem[BM90]{BM90}
\textsc{V. V. Batyrev, Y. I. Manin}, \textit{Sur le nombre des points rationnels
de hauteur born\'{e}e des vari\'{e}t\'{e}s alg\'{e}briques}.
Math. Ann., {\bf286}, 27--43 (1990).

\bibitem[BB06]{BB06}
\textsc{R. de la Bret\`{e}che and T. D. Browning}, \textit{Sums of arithmetic
functions over values of binary forms}. Acta Arith., {\bf125}, 291--304 (2006).

\bibitem[BB07]{BB07}
\textsc{R. de la Bret\`{e}che and T. D. Browning}, \textit{On
Manin's conjecture for singular del Pezzo surfaces of degree four,
I}. Michigan Mathematical Journal, {\bf55},  51--80 (2007).

\bibitem[BB08]{BB08}
\textsc{R. de la Bret\`{e}che and T. D. Browning}, \textit{Manin's conjecture
for quartic del Pezzo surfaces with a conic fibration}. Duke Math. J., {\bf160}, 1--69 (2011).


\bibitem[BBP10]{BBP10}
\textsc{R. de la Bret\`{e}che, T. D. Browning and E. Peyre}, \textit{On Manin's
conjecture for a family of Ch\^{a}telet surfaces}. Ann. of Math., to appear.
arXiv:1002.0255v1.

\bibitem[Bro07]{Bro07}
\textsc{T. D. Browning}, \textit{An overview of Manin's conjecture
for del Pezzo surfaces}. Analytic number theory - A tribute to Gauss
and Dirichlet (Goettingen, 20th June - 24th June, 2005), Clay
Mathematics Proceedings {\bf7}, 39--56 (2007).

\bibitem[Bro11]{Bro11}
\textsc{T. D. Browning}, \textit{The divisor problem for binary cubic forms}.
J. Th�orie Nombres Bordeaux, \emph{to appear}.

\bibitem[BHB07]{BHB07}
\textsc{T. D. Browning and D. R. Heath-Brown}, \textit{Counting rational points on hypersurfaces}.
J. Reine Angew. Math., {\bf584}, 83--115 (2007).

\bibitem[CT88]{CT88}
\textsc{D. F. Coray and M. A. Tsfasman}, \textit{Arithmetic on
singular Del Pezzo surfces}. Proc. London Math. Soc (3), {\bf57}(1)
, 25--87 (1988).

\bibitem[Dan99]{Dan99}
\textsc{S. Daniel}, \textit{On the divisor-sum problem
for binary forms}. J. Reine. Angew. Math., {\bf507}, 107--129 (1999).

\bibitem[Der07]{Der07}
\textsc{U. Derenthal}, \textit{On a constant arising in Manin's
Conjecture for Del Pezzo surfaces}. Math. Res. Letters, {\bf14}(3), 481--489 (2007).

\bibitem[DL10]{DL10}
\textsc{U. Derenthal and D. Loughran}, \textit{Singular del Pezzo
surfaces that are equivariant compactifications}. Proceedings of
Hausdorff Trimester on Diophantine equations in: Zapiski Nauchnykh
Seminarov (POMI), {\bf377}, 26--43 (2010).

\bibitem[FMT89]{FMT89}
\textsc{J. Franke, Y. I. Manin and Y. Tschinkel}, \textit{Rational
Points of Bounded Height on Fano Varieties}. Invent. Math. {\bf95}, 421--435 (1989).

\bibitem[Fra09]{Fra09}
\textsc{M. Franz.}, Convex - a Maple package for convex geometry. Version 1.1 2009.

\bibitem[HB84]{HB84}
\textsc{D. R. Heath-Brown}, \textit{Diophantine approximation with square-free numbers}.
Math. Zeit., {\bf187}, 335--344 (1984).

\bibitem[HB03]{HB03}
\textsc{D. R. Heath-Brown}, \textit{Linear relations amongst sums of two squares},
in Number theory and algebraic geometry, London Mathematical Society Lecture
Note Series, vol. 303 (Cambridge University Press, Cambridge, 2003), 133--176.

\bibitem[Lou10]{Lou10}
\textsc{D. Loughran} \textit{Manin's Conjecture for a Singular
Sextic Del Pezzo Surface}, J. Th\'{e}orie Nombres Bordeaux, {\bf22}, 675--701 (2010).

\bibitem[Man86]{Man86}
\textsc{Y. I. Manin}, \textit{Cubic Forms}. North-Holland
Mathematical Library {\bf4}, North-Holland Publishing Co., 2nd ed.
1986.

\bibitem[Mat11]{Mat11}
\textsc{L. Matthiesen}, \textit{Correlations of the divisor function}.
Proc. London Math. Soc., to appear. arXiv:1011.0019v1.

\bibitem[Pey95]{Pey95}
\textsc{E. Peyre}, \textit{Hauteurs et measures de Tamagawa sur les
vari\'{e}ti\'{e}s de Fano}. Duke Math. J., {\bf79}(1),
101--218 (1995).
\end{thebibliography}
\end{document}